\documentclass{compositio}

\usepackage[english]{babel}
\usepackage{amsmath,amssymb,amsthm}

\usepackage[T1]{fontenc}
\usepackage[utf8x]{inputenc}
\usepackage{enumitem}
\usepackage{verbatim}
\usepackage{graphicx}
\usepackage{faktor}
\usepackage{xcolor}
\usepackage{xfrac}
\usepackage{tikz,tikz-cd}
\usepackage{hyperref}
\usepackage[normalem]{ulem}
\usepackage{setspace}

\usepackage{calrsfs}
\DeclareMathAlphabet{\pazocal}{OMS}{zplm}{m}{n}

\newcommand{\tropC}{\scalebox{0.8}[1.3]{$\sqsubset$}}
\newcommand{\ttropC}{\widetilde{\tropC}}
\newcommand{\tropT}{\scalebox{0.9}[1.2]{$\top$}}
\newcommand{\tropD}{\scalebox{0.9}[1.3]{$\triangle$}}

\newcommand{\on}{\operatorname}

\newcommand{\PP}{\mathbb P}
\renewcommand{\k}{\mathbf k}
\newcommand{\m}{\mathfrak m}
\newcommand{\tR}{\widetilde{R}}
\newcommand{\tm}{\widetilde{\mathfrak m}}
\newcommand{\OO}{\mathcal O}
\renewcommand{\to}{\rightarrow}

\newcommand{\Aaff}{\mathbb A}

\newcommand{\oM}{\overline{\pazocal M}}

\newcommand{\tC}{\widetilde{\mathcal C}}
\newcommand{\cC}{\mathcal C}
\newcommand{\oC}{\overline{\mathcal C}}

\newcommand{\Gm}{\mathbb{G}_{\rm{m}}}
\newcommand{\Ga}{\mathbb{G}_{\rm{a}}}

\newcommand{\dvr}{\Delta}

\newcommand{\bq}{\begin{equation}}
\newcommand{\eq}{\end{equation}}
\newcommand{\ba}{\begin{aligned}}
\newcommand{\ea}{\end{aligned}}
\newcommand{\be}{\begin{enumerate}}
\newcommand{\ee}{\end{enumerate}}
\newcommand{\bsm}{\left(\begin{smallmatrix}}
\newcommand{\esm}{\end{smallmatrix}\right)}                   
\newcommand{\bpm}{\begin{pmatrix}}
\newcommand{\epm}{\end{pmatrix}}
\newcommand{\barr}{\begin{displaymath}\begin{array}{cccc}}
\newcommand{\earr}{\end{array}\end{displaymath}}
\newcommand{\barrl}{\begin{displaymath}\begin{array}{lcl}}
\newcommand{\earrl}{\end{array}\end{displaymath}}
\newcommand{\barl}{\begin{displaymath}\begin{array}{l}}
\newcommand{\earl}{\end{array}\end{displaymath}}
\newcommand{\bxym}{ \begin{displaymath}\xymatrix }
\newcommand{\exym}{\end{displaymath}}
\newcommand{\bcd}{\begin{center}\begin{tikzcd}}
\newcommand{\ecd}{\end{tikzcd}\end{center}}

\newcommand{\Pic}{\operatorname{Pic}}

\newcommand{\dist}{\operatorname{dist}}
\newcommand{\hhom}{\mathcal{H}\!om}
\newcommand{\Aut}{\operatorname{Aut}}
\newcommand{\Exc}{\operatorname{Exc}}
\newcommand{\lev}{\operatorname{lev}}
\newcommand{\id}{{\rm id}}

\theoremstyle{plain}
\newtheorem{thm}{Theorem}[section]
\newtheorem{lem}[thm]{Lemma}
\newtheorem{prop}[thm]{Proposition}
\newtheorem{cor}[thm]{Corollary}
\newtheorem*{theorem*}{Theorem}
\newtheorem*{fact}{Fact}

\theoremstyle{definition}
\newtheorem{definition}[thm]{Definition}

\theoremstyle{remark}
\newtheorem{rem}[thm]{Remark}
\newtheorem{exa}[thm]{Example}
\newtheorem*{caveat}{Caveat}
\newtheorem{claim}{Claim}

\usepackage{tikz,tikz-cd,tikz-3dplot}
\usepackage{pgfplots}
\usetikzlibrary{arrows,shadows,positioning, calc, decorations.markings, 
hobby,quotes,angles,decorations.pathreplacing,intersections,shapes}
\usepgflibrary{shapes.geometric}
\usetikzlibrary{fillbetween,backgrounds}

\tikzset{
  arrow/.pic={\path[tips,every arrow/.try,->,>=#1] (0,0) -- +(0,4pt);},
  pics/arrow/.default={triangle 90}
}

\tikzset{->-/.style={decoration={
  markings,
  mark=at position .6 with {\arrow{latex}}},postaction={decorate}}
  }

\tikzset{
  c/.style={every coordinate/.try}
}

\begin{document}

\title[Modular compactifications of $\pazocal{M}_{2,n}$]{Modular compactifications of $\pazocal{M}_{2,n}$ \\ with Gorenstein curves}
\author{Luca Battistella}
\email{lbattistella@mathi.uni-heidelberg.de}
\address{Ruprecht-Karls-Universität Heidelberg\\Im Neuenheimer Feld 205\\69120 Heidelberg\\Germany}
%
%\dedication{A dedication can be included here.}
\thanks{This project grew out of discussions with Francesca Carocci, whom I thank heartily. I am grateful to Daniele Agostini, Fabio Bernasconi, Sebastian Bozlee, Maria Beatrice Pozzetti, Dhruv Ranganathan, Luca Tasin, and Jonathan Wise for helpful conversations. I thank the anonymous referee, whose detailed comments were of great help in improving the exposition of this material. I thank the Max Planck Institute for Mathematics in Bonn for providing financial support and a stimulating research environment. During the revision of this paper, I was supported by the Deutsche Forschungsgemeinschaft (DFG, German Research Foundation) under Germany’s Excellence Strategy EXC-2181/1 - 390900948 (the Heidelberg STRUCTURES Cluster of Excellence).}

\begin{abstract}
\setstretch{1.1}{
We study the geometry of Gorenstein curve singularities of genus two, and of their stable limits. These singularities come in two families, corresponding to either Weierstrass or conjugate points on a semistable tail. For every $1\leq m <n$, a stability condition - using one of the markings as a reference point, and thus not $\mathfrak S_n$-symmetric - defines proper Deligne-Mumford stacks $\oM_{2,n}^{(m)}$ with a dense open substack representing smooth curves.}
\end{abstract}

\maketitle

\section{Introduction}
We construct alternative compactifications of the moduli stack of smooth $n$-pointed curves of genus two. The boundary of the Deligne-Mumford compactification, consisting of stable nodal curves, is gradually replaced by ever more singular curves, complying with more restrictive combinatorial requirements on the dual graph. For $1\leq m <n$, we introduce a notion of $m$-stability, that allows Gorenstein singularities of genus one and two while at the same time demanding that higher genus subcurves contain a minimum number of special points. Our main result concerning the stack of $m$-stable curves is the following:
\begin{theorem*}
 $\oM^{(m)}_{2,n}$ is a \emph{proper} irreducible Deligne-Mumford stack over $\operatorname{Spec}(\mathbb Z[\frac{1}{6}])$.
\end{theorem*}
This paper fits into the framework of alternative compactifications and birational geometry of the moduli space of curves, extending work of D.I. Smyth in genus one, but we expect it to find applications to enumerative geometry as well.

We classify Gorenstein singularities of genus two with any number of branches, and their (semi)stable models, highlighting the relation with Brill-Noether theory, and adopting the language of piecewise-linear functions on tropical curves. The key insight in defining the new stability conditions is that we can avoid non-Gorenstein singularities by modifying the curve at the conjugate point of the special branch; we use one of the markings to select the latter, and, more generally, to identify the $m$-stable limit in some very symmetric situations in which multiple choices are possible, a priori - as a result, our stability conditions are not $\mathfrak S_n$-symmetric.

We interpret crimping spaces (moduli of curves with a prescribed singularity type) as parameter spaces for the differential geometric data needed in order to construct a higher genus singularity from an ordinary $m$-fold point, and establish a connection with the existence of infinitesimal automorphisms, a phenomenon which had not fully emerged in lower genus.

Though a conspicuous amount of related research has been carried out on the birational geometry of $\oM_{2,n}$ for low values of $n$ \cite{Hassettg2,HL-tricanonical,Rulla,HL-birational_contraction, FedorchukGrimes,PolishchukJohnson}, this appears to be the first proposal of a sequence of modular compactifications for every $n$.

\subsection{From the Deligne-Mumford space to the Hassett-Keel program} One of the most influential results of modern algebraic geometry is the construction of a modular compactification of the stack of smooth pointed curves $\pazocal M_{g,n}$, due to P. Deligne, D. Mumford, and F. Knudsen, with the introduction of \emph{stable} pointed curves.

\begin{definition}\cite{DM}
 A connected, reduced, complete curve $C$ over an algebraically closed field $\k$, with distinct markings $(p_1,\ldots,p_n)$ lying in the smooth locus of $C$, is \emph{stable} if:
 \begin{enumerate}[leftmargin=.7cm]
  \item $C$ admits only nodes (ordinary double points) as singularities;
  \item every rational component of $C$ has at least three special points (markings or nodes), and every elliptic component has at least one.
 \end{enumerate}
\end{definition}

\begin{thm} \cite{DM,Knudsen}
 Assume $2g-2+n>0$. The moduli stack of stable pointed curves $\oM_{g,n}$ is a smooth and proper connected Deligne-Mumford stack over $\operatorname{Spec}(\mathbb Z)$, with projective coarse moduli space $\overline{\mathbf M}_{g,n}$, and normal crossing boundary representing nodal curves.
\end{thm}
On one hand, the Deligne-Mumford compactification has nearly every desirable property one could hope for; on the other, it is certainly not the unique modular compactification of $\pazocal M_{g,n}$. Classifying all of them is a challenging task, which was set out and partially performed in the inspiring work of Smyth \cite{SMY-towards} (see also \cite{Bozlee-thesis} for more recent efforts, bringing logarithmic geometry into the picture). The motivation comes mostly from birational geometry.

Even though the existence of $\overline{\mathbf M}_{g,n}$ can be deduced from nowadays standard theorems on stacks \cite{KM}, this moduli space was first constructed as a quotient, prompting the development of a powerful technique known as Geometric Invariant Theory  \cite{Gieseker,GIT,BalSwi}. Studying alternative compactifications of $\pazocal M_{g,n}$ sheds some light on the Mori chamber decomposition of $\overline{\mathbf M}_{g,n}$, and it is not by chance that the first steps in this direction were moved from a GIT perspective - by changing the invariant theory problem or the stability condition under consideration, and analysing the modular properties of the resulting quotients \cite{Schubert,Hassettg2,HassettHyeon}. This program, initiated by B. Hassett and S. Keel, aims to describe all the quotients arising in this way, and to determine whether every step of a log minimal model program for $\overline{\mathbf M}_{g,n}$ enjoys a modular interpretation in terms of curves with worse than nodal singularities \cite{CTV1,CTV2}. Since the early stages of this program, it has developed into a fascinating playground for implementing ideas that originated from (v)GIT into a general structure theory of Artin stacks \cite{AlperKresch,AFS1,AFS2,AFS3}. See for instance \cite{Morrison, FS} for more detailed and comprehensive accounts.

Only few steps of the Hassett-Keel program have been carried out in full generality. Yet, the program has been completed to a larger extent in low genus: with the introduction of Boggi-stable \cite{Boggi} and weighted pointed curves \cite{Hassettweighted} in genus zero, and with Smyth's pioneering work in genus one \cite{SMY1,SMY2,SMY3}, extending earlier work of D. Schubert. In a nutshell, an alternative compactification is defined by allowing a reasonably larger class of curve singularities (\emph{local condition}) while identifying their (semi)stable models, and disallowing the latter by imposing a stronger stability condition (\emph{global condition}, typically combinatorial); the valuative criterion ensures that the resulting moduli problem remains separated and universally closed.

A useful notion in this respect is that of the \emph{genus} of an isolated curve singularity: let $(C,q)$ be (the germ of) a reduced curve over an algebraically closed field $\k$ at its unique singular point $q$, with normalisation $\nu\colon\widetilde{C}\to C$ and $\mathcal F=\nu_*\OO_{\widetilde C}/\OO_C$, a skyscraper sheaf supported at $q$.
\begin{definition}\label{def:genus}\cite{SMY1}
If $C$ has $m$ branches (irreducible components of the normalisation) at $q$, and $\delta$ is the $\k$-dimension of $\mathcal F$, the genus of $(C,q)$ is defined as:
\[g=\delta-m+1.\] 
\end{definition}
The genus can be thought of as the number of conditions that a function must satisfy in order to descend from the seminormalisation (the initial object in the category of universal homeomorphisms $C^\prime\to C$, see \cite[\href{https://stacks.math.columbia.edu/tag/0EUS}{Tag 0EUS}]{stacks-project}, or a curve with the same topological space as $C$ and an ordinary $m$-fold point at $q$) to $C$. The node, for example, has genus zero (it coincides with its own seminormalisation). The genus of a singular point represents its non-topological contribution to the arithmetic genus of the curve containing it.

Smyth found that, for every fixed number $m$ of branches, there is a unique germ of Gorenstein singularity of genus one up to isomorphism, namely:
\begin{description}
 \item[$m=1$] the cusp, $V(y^2-x^3)\subseteq\Aaff^2_{x,y}$;
 \item[$m=2$] the tacnode, $V(y^2-yx^2)\subseteq\Aaff^2_{x,y}$;
 \item[$m\geq 3$] the union of $m$ general lines through the origin of $\Aaff^{m-1}$.
\end{description}
Singularities of this kind, with up to $m$ branches, together with nodes, form a deformation-open class of singularities. Moreover, the elliptic $m$-fold point can be obtained by contracting a smooth elliptic curve with $m$ rational tails in a one-parameter smoothing, and, roughly speaking, all stable models have a shape similar to this one.
\begin{definition}\cite{SMY1}
 A connected, reduced, complete curve $C$ of arithmetic genus one with smooth distinct markings $(p_1,\ldots,p_n)$ is \emph{$m$-stable}, $1\leq m<n$, if:
 \begin{enumerate}[leftmargin=0.7cm]
  \item it admits only nodes and elliptic $l$-fold points, $l\leq m$, as singularities;
  \item for every connected subcurve $E\subseteq C$ of arithmetic genus one, its \emph{level}:
  
  \noindent$\lvert E\cap\overline{C\setminus E}\rvert+\lvert\{i\colon p_i\in E\}\rvert$ is strictly larger than $m$;
  \item $H^0(C,\Omega_C^\vee(-\sum_i p_i))=0$ (finiteness of automorphism groups).
 \end{enumerate}
\end{definition}
The latter can be taken for a decency condition on the moduli stack. The first two, instead, are essential in guaranteeing the uniqueness of $m$-stable limits, as per the discussion above. Smyth's main result is the following.
\begin{thm}\cite{SMY1,SMY2}
 The moduli stack of $m$-stable curves $\oM_{1,n}(m)$ is a proper irreducible Deligne-Mumford stack over $\operatorname{Spec}\mathbb Z[1/6]$. It is \emph{not} smooth for $m\geq 6$. The coarse moduli spaces $\overline{\mathbf{M}}_{1,n}(m)$ arise as birational models of $\overline{\mathbf{M}}_{1,n}$ for the big line bundles $D(s)=s\lambda+\psi-\Delta$, where $\lambda$ is the Hodge class, $\psi$ is the sum of the $\psi$-classes, $\Delta$ is a boundary class, and there is an explicit relation between $s$ and $m$.
\end{thm}
Some further information on the geometry and singularities of these spaces (with the restriction $m=n-1$) has been discovered by Y. Lekili and A. Polishchuk in their study of \emph{strongly non-special} curves \cite{Lekili-Polishchuk}.

\subsection{Experimenting on a genus two tale}  In this subsection, we walk through the motivations and methods at the heart of our construction, exemplifying them in the simplest possible case, that of $\oM_{2,2}^{(1)}$. The facts we mention are either proved or explained in greater detail and generality in the paper. Here is a classical
\begin{fact}
 There are two unibranch singularities of genus two, the \emph{ramphoid cusp} or \emph{$A_4$-singularity} $V(y^2-x^5)\subseteq \Aaff^2_{x,y}$, and the ordinary genus two cusp $\operatorname{Spec}(\k[t^3,t^4,t^5])$. The former is Gorenstein, with stable model a Weierstrass tail (a genus two curve attached to a rational one at a Weierstrass point), while the latter is not Gorenstein, with stable model a non-Weierstrass tail of genus two.
\end{fact}
See Lemma \ref{lem:unibranch} and Proposition \ref{prop:tailI} below. Recall that every smooth curve of genus two is hyperelliptic, i.e. it can be realised as a two-fold cover of $\PP^1$, 
\begin{comment}ramified in six points\end{comment}
in a unique way up to projectivities. 
\begin{comment}; indeed, the unique $\mathfrak g^1_2$ is the complete canonical series\end{comment}
The cover automorphism is called the hyperelliptic involution $\sigma$; ramification points (fixed points of $\sigma$) are called Weierstrass, and in general $\{p,\sigma(p)\}$ are called conjugate points. See Section \ref{rmk:Wandconj}.

Let us try Smyth's approach out on genus two curves, starting with $\oM_{2,2}^{(1)}$. If we are going to require the level of a genus two subcurve to be at least two, it seems that we will need non-Gorenstein singularities in order to keep our moduli space proper. This might lead us into trouble; for example, the (log) dualising line bundle is classically exploited to construct canonical polarisations on stable curves, which in turn are essential in the proof that $\oM_{g,n}$ is an algebraic stack (or in the GIT construction of $\overline{\mathbf M}_{g,n}$). Yet, there is a way around the singularity $\k[\![t^3,t^4,t^5]\!]$.
\begin{fact}
 The $A_5$-singularity $V(y^2-yx^3)\subseteq\Aaff^2_{x,y}$ is a Gorenstein singularity of genus two with two branches. Its stable model is a genus two bridge, with conjugate attaching points. A marked union of two copies of $\PP^1$ along an $A_5$-singularity has no non-trivial automorphisms as soon as one of the two branches contains at least two markings.
\end{fact}
See Proposition \ref{prop:classification} and Corollary \ref{cor:explicitnoaut}. Going back to $\oM_{2,2}$, suppose $C$ is the nodal union of a genus two curve $Z$ with a rational tail $R$ supporting the two markings, so that $\lev(Z)=1$. If $R$ is attached to a Weierstrass point of $Z$, we may simply contract the latter (in a one-parameter smoothing), thus producing an irreducible ramphoid cusp with two markings. If instead $R$ is attached to a non-Weierstrass point $q_1$ of $Z$, we may blow-up the one-parameter family at the conjugate point $\sigma(q_1)$ in the central fibre, and then contract $Z$ to get a \emph{dangling} $A_5$-singularity (meaning that one of the branches is unmarked), which nonetheless has trivial automorphism group. We pursue this strategy, which makes our compactifications not semistable (see \cite[Definition 1.2]{SMY-towards} for the terminology). The necessity to include such curves was prefigured in \cite{AFSGm}.

To complete the picture, note that, in order to fix a deformation-open class of singularities, we need to allow cusps and tacnodes as well, due to the following
\begin{fact}
 The singularities appearing in the miniversal family of an $A_m$-singularity are all and only the $A_l$-singularities with $l\leq m$.
\end{fact}
See Theorem \ref{thm:ADE} for a more general statement - valid for all ADE singularities - due to A. Grothendieck. Since the semistable tail of a cusp (resp. tacnode) is an elliptic tail (resp. bridge), if we want our moduli space to remain separated, we should require that the level of a genus one subcurve be at least three at the same time as we introduce cusps and tacnodes. Hybrid situations may occur, such as an elliptic curve with a cusp, or an irreducible tacnode; since we need to allow a tacnode and a cusp sharing a branch, we should impose the level condition on genus one subcurves only when they are \emph{nodally attached}. Besides, in the latter example, we need to break the $\mathfrak S_2$-symmetry (relabelling the markings) in order to have a unique limit: we declare that $p_1$ must lie on the cuspidal branch. See Figure \ref{fig:one_tail_example}.

\begin{figure}[h]
 \includegraphics[width=\textwidth]{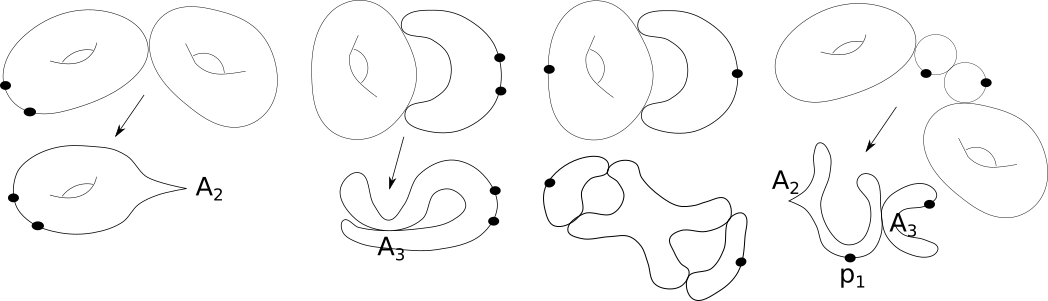}
  \caption{Examples of $2$-pointed stable curves and their $1$-stable counterparts.}\label{fig:one_tail_example}
\end{figure}

We are now in a position to cast a plausible definition of $\oM_{2,2}^{(1)}$.
\begin{definition}
 A connected, reduced, complete curve of arithmetic genus two $C$ over an algebraically closed field $\k$, with smooth and disjoint markings $(p_1,p_2)$, is \emph{$1$-stable} if:
 \begin{enumerate}[leftmargin=0.7cm]
  \item $C$ has only $A_1-,\ldots,A_4-$ and dangling $A_5-$ singularities.
  %\item Every subcurve of genus two has level $2$.
  \item $C$ coincides with its minimal subcurve of arithmetic genus two.
  \item A subcurve of arithmetic genus one is either nodally attached and of level three, or it is not nodally attached and it contains $p_1$.
  %\item There is no nodally attached subcurve of genus zero.
 \end{enumerate}
\end{definition}

The main result of the paper is that $\oM_{2,2}^{(1)}$ is a proper Deligne-Mumford stack, and the generalisation of this statement to an arbitrary number of markings and a range of stability conditions that we are going to discuss in the next sections.

Let us note in passing that the birational map $\oM_{2,2}\dashrightarrow\oM_{2,2}^{(1)}$ is not defined everywhere. The reason boils down to the following
\begin{fact}
 There is only one isomorphism class of $2$-pointed curves whose normalisation is $(\PP^1,q_1)\sqcup(\PP^1,q_2,p_1,p_2)$ and having an $A_5$-singularity at $q_1=q_2$. On the other hand, the moduli space of $2$-pointed irreducible curves of geometric genus zero with an $A_4$-singularity is isomorphic to $\Aaff^1$.
\end{fact}
The second statement can be motivated as follows: the pointed normalisation of such a curve is $(\PP^1,q,p_1,p_2)$, which has neither automorphisms, nor deformations. To produce an $A_4$-singularity at $q$ we may first collapse a non-zero tangent vector at $q$ (no choice involved), producing a cusp, and then collapse a line in the tangent space at the cusp, avoiding the support of its tangent cone $\ell$ (therefore, the moduli space is $\PP^1\setminus\{\ell\}=\Aaff^1$). See Lemma \ref{lem:crimping} and the discussion thereafter.

Let $\Delta=\Delta_{2,\emptyset|0,\{1,2\}}\subseteq\oM_{2,2}$ be the divisor of rational tails, and $\mathcal W\subseteq\oM_{2,2}$ the codimension two locus of Weierstrass tails. The $1$-stable limit of any point in $\Delta\setminus\mathcal W$ is the dangling $A_5$-singularity, while the $1$-stable limit of a Weierstrass tail is ill-defined (it depends on the choice of a $1$-parameter smoothing); we conjecture that the rational map (identity on the locus of smooth curves) admits a factorisation:
\bcd
& \operatorname{Bl}_{\mathcal W}(\oM_{2,2})\ar[dl]\ar[dr] & \\
\oM_{2,2}\ar[rr,dashed] & & \oM_{2,2}^{(1)}
\ecd
The blow-up should also encode enough information to contract an unmarked elliptic bridge to a tacnode. As it turns out, a modular desingularisation can be obtained by starting from the moduli space of pointed admissible covers, and performing a logarithmic modification based on some piecewise linear function on the tropicalization of the source curve. These methods have been developed in \cite{BC}, and we do not address the desingularisation here.

\subsection{Relation to other work} It would be interesting to compare $\oM_{2,2}^{(1)}$ explicitly with Smyth's $\oM_{2,2}(\pazocal Z)$ \cite{SMY-towards}, for the extremal assignment $\pazocal Z$ of unmarked subcurves; here we only note that, while the divisor $\Delta_{1,\{1\}|1,\{2\}}$ is contracted in $\oM_{2,2}^{(1)}$, the latter contains a copy of $\oM_{0,4}$ (see the third column, second row of Figure \ref{fig:one_tail_example}) that is replaced by the class of the rational $4$-fold point in $\oM_{2,2}(\pazocal Z)$. $\oM_{2,2}^{(1)}$ seems closely related to the space $\pazocal U^{ns}_{2,2}(ii)$ constructed in \cite{PolishchukJohnson}. More generally, it would be interesting to relate $\oM_{2,n}^{(m)}$ (for high values of $m$) to Polishchuk's moduli of curves with nonspecial divisors \cite{Polishchuk-nonspecial}. Finally, it seems plausible that $\oM_{2,n}^{(m)}$ (for low values of $m$) corresponds to a pointed variant of the spaces of admissible hyperelliptic covers with AD singularities constructed in \cite{Fedorchuk-hyperellipticAD}.

\subsection{Outline of results and plan of the paper} In Section \ref{sec:sing} we classify all the Gorenstein curve singularities of genus two. They come in two families: the first one ($I$) includes the ramphoid cusp, the $D_5$-singularity, and for $m\geq 3$ the union of \emph{a singular branch} (a cusp) and $m-1$ lines living in $\Aaff^m$. The second one ($I\!I$) includes the $A_5$- and $D_6$-singularities, and for $m\geq 4$ the union of \emph{two tangent branches} (forming a tacnode) with $m-2$ lines in $\Aaff^{m-1}$. See Proposition \ref{prop:classification}.

In Section \ref{sec:crimp} we translate the condition that a complete pointed curve of genus two has no infinitesimal automorphisms into a \emph{mostly} combinatorial criterion. For every fixed number of branches $m$ and genus two singularity type $\in\{I,I\!I\}$, there are two isomorphism classes of pointed curves whose normalisation is $\bigsqcup_{i=1}^m(\PP^1,q_i,p_i)$ and having a singularity of the prescribed type at $q$; one of them has $\Aut(C,p)=\Gm$, while the other one has trivial automorphism group. This phenomenon is a novelty to genus two. We take a detour into moduli spaces of singularities to justify the claim, and explain how to interpret the \emph{crimping spaces} geometrically in terms of the information we need to construct a genus two singularity from a (non-Gorenstein) singularity of lower genus. This is not strictly necessary in what follows, since the singularity with one-pointed branches never satisfies the level condition we demand from our curves, yet this description is useful in analysing the indeterminacy of $\oM_{2,n}^{(m_1)}\dashrightarrow\oM_{2,n}^{(m_2)}$.

In Section \ref{sec:sstails} we study the (semi)stable limits; starting from a $1$-parameter family of semistable curves with smooth generic fibre and regular total space, we show that the shape of a subcurve of the central fibre that can be contracted into a Gorenstein singularity is strongly constrained. Singularities of type $I$ arise when the special branch (corresponding to the cusp in the contraction) is attached to a \emph{Weierstrass point} of the minimal subcurve of genus two (the \emph{core}), while singularities of type $I\!I$ occur when the special branches (corresponding to the tacnode in the contraction) are attached to \emph{conjugate points}. Furthermore, the size of the curve to be contracted only depends on one number - roughly speaking, the distance of the special branches from the core. The first statement is a consequence of the following simple observation: if $\phi\colon\widetilde{\mathcal C}\to\mathcal C$ is a contraction to a family of Gorenstein curves, $\phi^*\omega_{\mathcal C}$ is trivial on a neighbourhood of the exceptional locus of $\phi$, and it coincides with $\omega_{\widetilde{\mathcal C}}$ outside it. Now, whereas the dualising line bundle of a Gorenstein curve of genus one with no separating nodes is trivial (see \cite[Lemma 3.3]{SMY1}) and all smooth points display the same behaviour (in the sense that they are non-special), the simplest instance of Brill-Noether theory manifests itself in genus two, with the distinction between Weierstrass and non-Weierstrass points, and the expression $\omega_Z=\OO_Z(q+\sigma(q))$. The correct extension of these concepts to nodal curves was formulated in the '80s within the theory of admissible covers and limit linear series. We phrase the shape restrictions in terms of the existence of a certain piecewise-linear function on the dual graph of the central fibre.

In Section \ref{sec:stability} we define the notion of \emph{$m$-stable} $n$-pointed curve of genus two, for every $1\leq m<n$. The basic idea is to trade worse singularities - of both genus one and two, bounded by $m$ in the sense of the embedding dimension - with more constraints on the combinatorics of the dual graph - the \emph{level} condition, which bounds below in terms of $m$ the number of special points (nodes and markings) that any subcurve of genus one or two has to contain. On the other hand, it is already clear from the discussion above that we need to break the $\mathfrak S_n$-symmetry, in order to write the dualising line bundle of the minimal subcurve of genus two as $\OO_Z(q_1+\sigma(q_1))$, in other words to choose which branches of a semistable model are to be dubbed special. We do so by using the first marking as a reference point, so that $q_1$ comes to denote the point of $Z$ closest to $p_1$. This shapes our algorithm to construct the $m$-stable limit of a given $1$-parameter smoothing. Unavoidably, the formulation of the stability condition is slightly involved, including a prescription of the interplay between $p_1$ and the singularity. We prove that the moduli stack of $m$-stable curves is algebraic, and it satisfies the valuative criterion of properness.

\subsection{Future directions of work} Besides regarding this paper as a case-study of the birational geometry of moduli spaces of curves, it also has some nontrivial applications to Gromov-Witten theory. We set up some questions we would like to come back to in future work.
\begin{enumerate}[leftmargin=.7cm]
 \item The indeterminacy of the rational map $\oM_{2,n}^{(m_1)}\dashrightarrow\oM_{2,n}^{(m_2)}$ can be resolved modularly: a space dominating all the $\oM_{2,n}^{(m)}$ can be obtained as a logarithmic modification (as in \cite{RSPW1}) of the space of \emph{admissible covers} (of degree two, with rational target and six ramification points). We shall describe this construction in more details in a forthcoming paper. We wonder whether the models constructed here correspond to the trace of the minimal model program on a two-dimensional slice of the cone of pseudo-effective divisors, as in \cite{SMY2}.
 
 Recently, S. Bozlee, B. Kuo, and A. Neff have classified all the compactifications of $\pazocal{M}_{1,n}$ in the stack of Gorenstein curves with distinct markings \cite{BKN} - it turns out that there are many more than envisioned by Smyth, although the numerosity arises more from combinatorial than geometric complications. The techniques developed in \cite{BC} suggest that more birational models of $\pazocal M_{2,n}$ could be constructed by allowing non-reduced Gorenstein curves as well. It would be interesting if all compactifications of $\pazocal{M}_{2,n}$ could be classified by a mixture of our techniques. More generally, we could ask the same question about the moduli space of hyperelliptic curves, although it is known not to be a Mori dream space \cite{BarrosMullane}.
 
 %More generally, a question outstanding to our knowledge is whether the whole program fits in the framework of stability developed in \cite{DHLinstability}.
 
 \item Enumerative geometry: the link between reduced Gromov-Witten invariants in genus one (see for example \cite{VZ,Zingerred,LZ}) and maps from singular curves (see \cite{VISC}) was partially uncovered in \cite{BCM}, and brought in plain view by \cite{RSPW1,RSPW2}. In joint work with F. Carocci \cite{BC}, we exploit similar techniques to desingularise the main component of the space of genus two maps to projective space. We enrich the logarithmic structure by including a compatible admissible cover. A universal morphism to a Gorenstein curve is constructed on a logarithmically \'etale model of the base, encoding the choice of a \emph{tropical canonical divisor}. We stress the fact that non-reduced fibres (\emph{singular ribbons}) arise naturally in that context. The main component is recovered as those maps that factor through the Gorenstein contraction. Our desingularisation is less efficient than \cite{HLN}, but maps from singular curves provide a conceptual definition of reduced invariants for projective complete intersections and beyond. We hope that they will make comparison results (standard vs. reduced) easier to prove. This would lead to a modular interpretation of Gopakumar-Vafa invariants \cite{Pandha}.
\end{enumerate}

\section{Gorenstein curve singularities of genus two and their dualising line bundles}\label{sec:sing}

We produce an algebraic classification of the (complete) local rings of Gorenstein curve singularities of genus two over an algebraically closed field $\k$. The proof involves a technical calculation with the conductor ideal. Alternatively, one can look for a local generator of the dualising line bundle at the singularity; we remark on this below.

Let $(C,q)$ be the germ of a reduced curve singularity, and let $(R,\m)$ denote $(\hat\OO_{C,q},\m_q)$, with normalisation $(\tR,\tm)\simeq\left(\k[\![t_1]\!]\oplus\ldots\oplus\k[\![t_m]\!],\langle t_1,\ldots,t_m\rangle\right)$.
Here $m$ is the number of branches of $C$ at $q$. Recall the Definition \ref{def:genus} of the genus:
\[g=\delta-m+1;\]
so, for genus two, $\delta=m+1$. Following \cite[Appendix A]{SMY1}, we consider $\tR/R$ as a $\mathbb Z$-graded module with:
\[ (\tR/R)_i:=\tm^i/(\tm^i\cap R)+\tm^{i+1};\]
furthermore, adapting Smyth's remarks in \emph{loc. cit.} to our situation:
\begin{enumerate}
\item $m+1=\delta(p)=\sum_{i\geq 0}\dim_\k(\tR/R)_i;$
\item $2=g=\sum_{i\geq 1}\dim_\k(\tR/R)_i;$
\item\label{obs:add} if $(\tR/R)_i=(\tR/R)_j=0$ then $(\tR/R)_{i+j}=0$.
\end{enumerate}
We will also make use of the following observations:
\begin{enumerate}[resume]
 \item\label{obs:igrad} $\sum_{i\geq j}(\tR/R)_i$ is a grading of $\tm^j/(\tm^j\cap R)$;
 \item\label{obs:ses} there is an exact sequence of $R/\m=\k$-modules:
 \[ 0\to A_i:=\frac{\tm^i\cap R}{\tm^{i+1}\cap R}\to \frac{\tm^i}{\tm^{i+1}}\to \left(\tR/R\right)_i\to 0\]
\end{enumerate}
\begin{lem}\label{lem:unibranch}
 There are two unibranch curve singularities of genus two; only one of them is Gorenstein, namely the $A_4$-singularity or \emph{ramphoid cusp}: $V(y^2-x^5)\subseteq\Aaff^2_{x,y}$.
\end{lem}
\begin{proof}
 In the unibranch case $\dim_\k(\tR/R)_1\leq 1$, hence equality holds (by observation \eqref{obs:add} above). We are left with two cases:
 \begin{itemize}[leftmargin=15pt]
  \item Either $\dim_\k(\tR/R)_2=1$ and $\dim_\k(\tR/R)_i=0$ for all $i\geq 3$: in this case $\tm^3\subseteq\m$ by observation \eqref{obs:igrad}. From \eqref{obs:ses} we see that $\tm^3=\m$, hence $R\simeq\k[\![t^3,t^4,t^5]\!],$ a non-Gorenstein singularity sitting in $3$-space, which is obtained by collapsing a second-order infinitesimal neighbourhood of the origin in $\Aaff^1$ (we shall call it an ordinary cusp of genus two).
  
  \item Or $\dim_\k(\tR/R)_3=1$ and $\dim_\k(\tR/R)_i=0$ for $i=2$ and for all $i\geq 4$: in this case $\tm^4\subseteq\m$ by observation \eqref{obs:igrad}. On the other hand from $\dim_\k(\tm^2\cap R/\tm^3\cap R)=1$ we deduce that there is a generator of degree $2$, and from $\dim_\k(\tm^3\cap R/\tm^4\cap R)=0$ there is none of degree $3$. We may write the generator as $x=t^2+ct^3$, and $\m=\langle x\rangle+\tm^4$. Up to a coordinate change (i.e. automorphism of $\k[\![t]\!]$), we may take $x=t^2$, and \[\m/\m^2=\langle t^2,t^5\rangle,\] so $R\simeq\k[\![x,y]\!]/(x^5-y^2)$, as anticipated.
 \end{itemize}
\end{proof}

From now on, we only look for Gorenstein singularities. With notation as above, let $I=(R:\tilde R)=\operatorname{Ann}_R(\tilde R/R)$ be the \emph{conductor ideal} of the singularity. Recall e.g. \cite[Proposition VIII.1.16]{AK}: $(C,q)$ is Gorenstein if and only if
\[\dim_\k(R/I)=\dim_\k(\tR/R)(=\delta).\]

Recall from \cite[Definition 2-1]{Stev} that a curve singularity $(C,q)$ is \emph{decomposable} if $C$ is the union of two curves $C_1$ and $C_2$ that lie in distinct smooth spaces intersecting each other transversely in $q$. With a parametrisation \[\k[\![x_1,\ldots,x_l]\!]\to\k[\![t_1]\!]\oplus\ldots\oplus\k[\![t_m]\!]\] given by $x_i=x_i(t_1,\ldots,t_m)$, being decomposable means that there is a partition $S_0\sqcup S_1=\{1,\ldots,m\}$ such that for every $i\in\{1,\ldots,l\}$ there exists a $j\in\{0,1\}$ such that $x_i$ does not depend on any $t_s$ for $s\in S_{1-j}$. Aside from the node, Gorenstein singularities are never decomposable \cite[Proposition 2.1]{AFSGm}.

\begin{prop}\label{prop:classification}
 For every fixed integer $m\geq 2$, there are exactly two Gorenstein curve singularities of genus two with $m$ branches.
\end{prop}
\begin{proof}
 We only need to find a basis for $\m/\m^2$, because a map of complete local rings that is surjective on cotangent spaces is surjective. From observation \eqref{obs:add} again, we find three possibilities for the vector $(d_1,d_2,d_3)$, $d_i=\dim_{\k}(\tR/R)_i$; $d_{\geq 4}=0$ in any case.
 
 \smallskip
 
 \textbf{Case} $(2,0,0)$. We see that $\tm^2\subseteq I$, so, if $(C,q)$ were Gorenstein, \eqref{obs:ses} would imply: \[m+1=\delta=\dim_\k(R/I)\leq \dim_\k(R/\tm^2)=\dim_\k A_0+\dim_\k A_1=1+(m-2)=m-1,\] a contradiction. Note: the singularity turns out to be decomposable in this case.
 
 \smallskip
 
 \textbf{Case} $(1,1,0)$. We have $\tm^3\subseteq I$. We are going to write down the $m-1$ generators of $A_1 \pmod {\tm^3}$. To express them in the simplest possible form, we perform at first only polynomial manipulations of the generators, while changing coordinates on the normalisation only at the end; the first step gives us the Zariski-local classification, which will be useful in the next section, while the second step completes the \'etale local (or formal) classification that we are interested in at the moment. Note that there is a short exact sequence:
 \[0\to{\tm^2\cap R}/{\m^2}\to{\m}/{\m^2}\to A_1\to 0\]
 
 The first generator, call it $x_1$, has a non-trivial linear term in at least one of the variables, say $t_1$. By scaling $x_1$ and possibly adding a multiple of $x_1^2$, we can make it into the form:
 $x_1=t_1\oplus p_{1,2}(t_2)\oplus\ldots\oplus p_{1,m}(t_m) \pmod{\tm^3}.$ Now we can use $x_1$ and $x_1^2$ to make sure the second generator does not involve $t_1$ at all. It will still have a linear term independent of $t_1$, say non-trivial in $t_2$. By scaling and adding a multiple of $x_2^2$, we can write $x_2=0\oplus t_2\oplus\ldots\oplus p_{2,m}(t_m) \pmod{\tm^3}.$ By taking a linear combination of $x_1$ with $x_2$ and $x_2^2$, we may now reduce $x_1$ to the form $t_1\oplus0\oplus p_{1,3}(t_3)\oplus\ldots\oplus p_{1,m}(t_m)\pmod{\tm^3}$. Continuing this way, by Gaussian elimination with the generators and their squares, we may write them as:
 \begin{align*}
  x_1= & t_1\oplus0\oplus\ldots\oplus\alpha_{1,m}t_m+\beta_{1,m}t_m^2\\
  x_2= & 0\oplus t_2\oplus\ldots\oplus\alpha_{2,m}t_m+\beta_{2,m}t_m^2\\
  &\ldots\\
  x_{m-1}= & 0\oplus\ldots\oplus t_{m-1}\oplus\alpha_{m-1,m}t_m+\beta_{m-1,m}t_m^2 \pmod{\tm^3}.
 \end{align*}
 If $x_i\in I$ for some $i$, then $t_i\in R$, and the singularity would be decomposable. So, by the Gorenstein condition, $R/I$ is generated by $1,x_1,\ldots,x_{m-1},$ and an extra element $y$. Hence $x_i^2\in I$ for all but at most one $i$.
 
 Suppose first that $x_i^2$ belongs to $I$ for all $i\in\{1,\ldots,m-1\}$. Note that $t_m^2$ cannot lie in $R$ since $\dim_{\k}(\tR/R)_2=1$. Then the conductor ideal is $I=\langle t_1^2,\ldots,t_{m-1}^2,t_m^3\rangle$, so $\dim_\k(\tR/I)=2m+1$, which contradicts the Gorenstein condition $\dim_\k(\tR/I)=2\delta$.
 
 Therefore, there is an $i$ such that $x_i^2\notin I$, say $i=1$. Then $t_i^2\in I\subseteq R$ for $i=2,\ldots,m-1$. If $\alpha_{i,m}\neq 0$ for some $i$ in this range, then $t_m^2\in R$ as well, so $t_1^2=x_1^2-O(t_m^2)\in R$, contradicting $\dim_\k(\tR/R)_2=1$. Therefore $\alpha_{i,m}=0$ for $i\in\{2,\ldots,m-1\}$. If $\alpha_{1,m}=0$, we would have $t_1^2=x_1^2-O(t_m^4)\in I\subseteq R$, so $x_1^2\in I$ as well, which is a contradiction. We are reduced to the following expression:
 \begin{align}\label{coordII-cs}
 \begin{split}
  x_1= & t_1\oplus0\oplus\ldots\oplus\alpha_{1,m}t_m+\beta_{1,m}t_m^2\\
  x_2= & 0\oplus t_2\oplus\ldots\oplus\beta_{2,m}t_m^2\\
  &\ldots\\
  x_{m-1}= & 0\oplus\ldots\oplus t_{m-1}\oplus \beta_{m-1,m}t_m^2 \pmod{\tm^3},
 \end{split}
 \end{align}
 with $\beta_{1,m}\in\k$ and $\alpha_{1,m},\beta_{i,m}\in\k^\times,\ i=2,\ldots,m-1$ (by indecomposability). Finally, we change coordinates in $t_m$ (abusing notation, $t_m:=\alpha_{1,m}t_m+\beta_{1,m}t_m^2$) and rescale the other $t_i$ to obtain:
 \begin{align}\label{coordII}
 \begin{split}
  x_1= & t_1\oplus0\oplus\ldots\oplus t_m\\
  x_2= & 0\oplus t_2\oplus\ldots\oplus t_m^2\\
  &\ldots\\
  x_{m-1}= & 0\oplus\ldots\oplus t_{m-1}\oplus t_m^2\pmod{\tm^3}.
 \end{split}
 \end{align}
 We check that $R/I=\langle 1,x_1,\ldots,x_{m-1},x_1^2\rangle$ and $\tR/R$ is of type $(1,1,0)$. In case $m=2$, we need an extra generator $y=t_2^3$. Equations are given by:
 \begin{itemize}
  \item $y(y-x_1^3)$ if $m=2$ ($A_5$-singularity or \emph{oscnode});
  \item $x_1x_2(x_2-x_1^2)$ if $m=3$ ($D_6$-singularity);
  \item $\langle x_3(x_1^2-x_2),x_i(x_j-x_k)\rangle_{1\leq i<j<k\leq m-1 \text{ or }1<j<k<i\leq m-1}$ if $m\geq 4$.
 \end{itemize}

 \smallskip
 
 \textbf{Case} $(1,0,1)$. We have $\tm^4\subseteq I$. By an argument similar to the above one, we write generators for $A_1$ as $x_i=\ldots\oplus t_i\oplus\ldots \oplus\alpha_{i,m}t_m+\beta_{i,m}t_m^2+\gamma_{i,m}t_m^3$, for $i=1,\ldots,m-1$.
 Then $R/I=\langle 1,x_1,\ldots,x_{m-1},y\rangle$. For all but at most one $i$, $x_i^2\in I$, but definitely $x_i^3\in I$ for all $i$. On the other hand $t_m^3\notin R$, because otherwise $t_i^3=x_i^3-\alpha_{i,m}^3t_m^3+O(t_m^4)$ would belong to $R$ as well, contradicting $\dim_\k(\tR/R)_3=1$. From this we deduce that $\alpha_{i,m}=0$ for all $i=1,\ldots,m-1$. Since $\dim_\k(\tR/R)_2=0$, there has to be another generator of $\m/\m^2$ of degree two in $t_m$, which we may write as $x_m=t_m^2+\gamma_{m,m}t_m^3$. We can use $x_m$ to remove all the $t_m^2$ pieces from $x_1,\ldots,x_{m-1}$, so we are reduced to the following expression:
  \begin{align}\label{coordIII-cs}
 \begin{split}
  x_1= & t_1\oplus0\oplus\ldots\oplus \gamma_{1,m}t_m^3\\
  x_2= & 0\oplus t_2\oplus\ldots\oplus \gamma_{2,m}t_m^3\\
  &\ldots\\
  x_{m-1}= & 0\oplus\ldots\oplus t_{m-1}\oplus \gamma_{m-1,m}t_m^3\\
  x_m= & 0\oplus\ldots\oplus t_m^2+\gamma_{m,m}t_m^3 \pmod{\tm^4},
  \end{split}
 \end{align}
 with $\gamma_{m,m}\in\k$ and $\gamma_{i,m}\in\k^\times,\ i=1,\ldots,m-1$ (by indecomposability). Finally, we change coordinates in $t_m$ (abusing notation $t_m:=t_m\sqrt{1+\gamma_{m,m}t_m}$)\footnote{For this to be possible, we have to assume that $\k$ has characteristic different from $2$.} and rescale the other $t_i$ to obtain:
 \begin{align}\label{coordIII}
 \begin{split}
  x_1= & t_1\oplus0\oplus\ldots\oplus t_m^3\\
  x_2= & 0\oplus t_2\oplus\ldots\oplus t_m^3\\
  &\ldots\\
  x_{m-1}= & 0\oplus\ldots\oplus t_{m-1}\oplus t_m^3\\
  x_m= & 0\oplus\ldots\oplus t_m^2 \pmod{\tm^4}.
  \end{split}
 \end{align}
 We check that $R/I=\langle 1,x_1,\ldots,x_{m-1},x_m\rangle$ and $\tR/R$ is of type $(1,0,1)$. Incidentally, when $m=1$, we recover the unique Gorenstein singularity of Lemma \ref{lem:unibranch}. Equations are given by:
 \begin{itemize}
  \item $x^5-y^2$ if $m=1$ ($A_4$-singularity or \emph{ramphoid cusp}, with $x=t^2,y=t^5$);
  \item $y(y^3-x^2)$ if $m=2$ ($D_5$-singularity, with $x=x_1,y=x_2$);
  \item $\langle x_3(x_1-x_2),x_3^3-x_1x_2\rangle$ if $m=3$;
  \item $\langle x_i(x_j-x_k), x_m(x_i-x_j),x_m^3-x_1x_2\rangle_{i,j,k\in\{1,\ldots,m-1\}\text{ all different}}$ if $m\geq 4$.
 \end{itemize}
\end{proof}

%\begin{rem}
 %Not-necessarily Gorenstein singularities can be obtained by gluing various Gorenstein singularities of genus $\leq 2$ along subschemes of length $\leq 3$. Classifying all of them would not necessarily be easy.
%\end{rem}

\begin{rem}
 We sketch an alternative proof of the above proposition based on meromorphic differentials, see also Corollary \ref{cor:dualising_line_bundle} below. We address the case $(1,0,1)$ and leave the (easier) case $(1,1,0)$ to the interested reader. The setup is as in \cite[\S 2.1]{RSPW2}: let $C$ be a projective Gorenstein curve with a unique singularity of genus two at the point $q$. Let:
 \[ \widetilde C\xrightarrow{\nu}\widehat C\xrightarrow{\mu} C\]
 be respectively the normalisation and semi-normalisation of $C$. We have inclusions:
 \begin{align*}
  {} & {} & \OO_C & \subseteq & \mu_* \OO_{\widehat C} & \subseteq & \mu_*\nu_* \OO_{\widetilde C} & \subseteq  & K \\
  J  & \supseteq & \omega_C & \supseteq & \mu_* \omega_{\widehat C} & \supseteq & \mu_*\nu_* \omega_{\widetilde C} & {}
 \end{align*}
where $K$ is the sheaf of rational functions, and $J$ the sheaf of meromorphic differentials. The rows are dual to each other with respect to the residue pairing $J\otimes K\to\k$ \cite[Proposition 1.16(ii)]{AK}. The skyscraper sheaf $\omega_{\widehat C}/\nu_* \omega_{\widetilde C}$ is generated by the logarithmic differentials:
\begin{equation}\label{eqn:logdifferentials}\frac{\operatorname{d} t_i}{t_i}-\frac{\operatorname{d} t_j}{t_j},\quad i,j\in\{1,\ldots,m\}.\end{equation}
The skyscraper sheaves $\mu_*\OO_{\widehat C}/\OO_C$ and $\omega_C/\mu_*\omega_{\widehat C}$ have length two. Let $\eta_1$ be a generator of the latter; since $C$ is Gorenstein, we may assume that $\eta_1$ is a local generator of $\omega_C$. Since $\tm^4\subseteq R$ and by Equation \eqref{eqn:logdifferentials}, we may assume that $\eta_1$ takes the following form:
\begin{equation}\label{eqn:eta1}\eta_1=\zeta\frac{\operatorname{d} t_1}{t_1}+\sum_{i=1}^m \alpha_i \frac{\operatorname{d} t_i}{t_i^2}+ \sum_{i=1}^m \beta_i \frac{\operatorname{d} t_i}{t_i^3}+ \sum_{i=1}^m \gamma_i \frac{\operatorname{d} t_i}{t_i^4}.\end{equation}
Since the constant functions descend to $C$, the residue condition implies that $\zeta=0$. 

Not all the $\gamma_i$ can be zero, otherwise we would have $\tm^3\subseteq R$, so say $\gamma_m\neq 0$; this implies that $t_m^3\notin R$ (and in particular $t_m\notin R$). Up to scaling we have $\gamma_m=1$.

Since $A_2=0$, there is a linear combination:
\[ q=t_m^2+yt_m^3\in R.\]
Pairing with $\eta_1$ we find $y=-\beta_m$.

Since $A_1$ has dimension one, for every $i=1,\ldots,m-1$ there is a linear combination:
\[ l_i=t_i+x_it_m+y_it_m^2+z_it_m^3\in R.\]
Subtracting a scalar multiple of $q$, we may assume that $y_i=0$ for all $i$. If $x_i$ were not zero for any $i$, we would have an element $ql_i=x_it_m^3+O(t_m^4)$ in $R$, against the assumption that $t_m^3\notin R$; hence $x_i=0$ for all $i$ as well. Pairing with $\eta_1$ we find $z_i=-\alpha_i$. So $\alpha_i\neq0$ for $i=1,\ldots,m-1$, otherwise $C$ would be decomposable. 

\noindent Taking $l_i^2$, we have $t_i^2\in R, i=1,\ldots,m-1,$ as well. Pairing with $\eta_1$, we find $\beta_i=0$.

\noindent Taking $l_i^3$, we have $t_i^3\in R, i=1,\ldots,m-1,$ as well. Pairing with $\eta_1$, we find $\gamma_i=0$.

Up to elements of $\mu_*\nu_*\omega_{\widetilde C}$, we may therefore write:
\[\eta_2:= q\eta_1=\frac{\operatorname{d} t_m}{t_m^2}.\]

Finally, subtracting a multiple of $\eta_2$, Equation \eqref{eqn:eta1} becomes:
\[\eta_1=\sum_{i=1}^{m-1} \alpha_i \frac{\operatorname{d} t_i}{t_i^2}+ \beta_m \frac{\operatorname{d} t_m}{t_m^3}+  \frac{\operatorname{d} t_m}{t_m^4}, \text{ with } \alpha_i\in\k^\times,\beta_m\in\k.\]

\end{rem}

\begin{definition}\label{def:special_branches}
 In case $(1,0,1)$, we say the singularity is \emph{of type $I$}, and the branch parameterised by $t_m$ is called \emph{singular}; in case $(1,1,0)$, we say the singularity is \emph{of type $I\!I$}, and the branches parameterised by $t_1$ and $t_m$ are called \emph{twin}. We shall refer to the singular or twin branches as \emph{special} or \emph{distinguished}; all other branches are \emph{axes}. \emph{Branch} remains a generic name, indicating any of the previous ones.
\end{definition}

We gather a description of the dualising line bundle in the following:
\begin{cor}\label{cor:dualising_line_bundle}
 Let $\nu\colon C\to \bar C$ be the normalisation of a Gorenstein singularity of genus two, with $\nu^{-1}(q)=\{q_1,\ldots,q_m\}$.
 \begin{enumerate}[label=(\Roman*)]
 
 \item With local parametrisation as in \eqref{coordIII}, $\omega_{\bar C}$ is generated by: \[\frac{\operatorname{d}t_1}{t_1^2}+\ldots+\frac{\operatorname{d}t_{m-1}}{t_{m-1}^2}-\frac{\operatorname{d}t_m}{t_m^4},\]
 and $\nu^*\omega_{\bar C}=\omega_C(2q_1+\ldots+2q_{m-1}+4q_m)$.

 \item With local parametrisation as in \eqref{coordII}, $\omega_{\bar C}$ is generated by:
  \[\frac{\operatorname{d}t_1}{t_1^3}+\frac{\operatorname{d}t_2}{t_2^2}+\ldots+\frac{\operatorname{d}t_{m-1}}{t_{m-1}^2}-\frac{\operatorname{d}t_m}{t_m^3},\]  
  and $\nu^*\omega_{\bar C}=\omega_C(3q_1+2q_2+\ldots+2q_{m-1}+3q_m)$.
                                                                                                                                  \end{enumerate}
\end{cor}

\section{Tangent sheaf and automorphisms}\label{sec:crimp}
In this section we analyse the tangent sheaf of a genus two singularity. For a complete pointed Gorenstein curve of genus two, we translate the absence of infinitesimal automorphisms into a (mostly) combinatorial criterion. We will use this in Section \ref{sec:stability}, when we define stability conditions on the stack of pointed Gorenstein curves of genus two, to make sure that the resulting substacks are Deligne-Mumford.

In the first lemma, we find conditions for a vector field on the normalisation, vanishing at the preimage of the singular point, to descend to the singular curve. We do so by an explicit computation in the Zariski-local coordinates of the previous section.

\begin{lem}\label{lem:aut}
Let $(C,q)$ be a Gorenstein curve singularity of genus two, with pointed normalisation $\nu\colon(\tilde C,\{q_i\}_{i=1,\ldots,m})\to (C,q)$, and assume $\operatorname{char}(\k)\neq 2,3$. Then $\Omega_C^\vee$ in $\nu_*\Omega^\vee_{\tilde C}\otimes K(\tilde C)$ satisfies:
\begin{itemize}
 \item it is contained in $\Omega_{-1}:=\nu_*\Omega^\vee_{\tilde C}(-\sum q_i)$;
 \item it contains $\Omega_{-3}:=\nu_*\Omega^\vee_{\tilde C}(-\sum3q_i)$;
 \item and its image in $\Omega_{-1}/\Omega_{-3}=\bigoplus \nu_*\Omega^\vee_{\tilde C}(-q_i)_{|2q_i}$ is a half-dimensional subspace, of which we give explicit equations in local coordinates.
\end{itemize}
\end{lem}
\begin{comment}
\begin{lem}\label{lem:aut}
Let $(C,q)$ be a Gorenstein curve singularity of genus two, with pointed normalisation $\nu\colon(\tilde C,\{q_i\}_{i=1,\ldots,m})\to (C,q)$, and assume $\operatorname{char}(\k)\neq 2,3$. There are exact sequences of sheaves
\bcd
0\ar[r]  & \nu_*\Omega_{\tilde C}^\vee(-\sum_i 3q_i)\ar[r]  &\nu_*\Omega_{\tilde C}^\vee(-\sum_i q_i)\ar[r]\ar[dr,phantom,"\Box"]  &\nu_*\bigoplus_i\Omega_{\tilde C}^\vee(-q_i)_{|2q_i}\ar[r]  & 0 \\
0\ar[r]  &\nu_*\Omega_{\tilde C}^\vee(-\sum_i 3q_i)\ar[r]\ar[u,equal]  & \Omega_C^\vee\ar[r]\ar[u] & \k^{\oplus m}\ar[r]\ar[u,hook,"\phi"] & 0
\ecd
 The image of the rightmost vertical map admits an explicit description in local coordinates.
\end{lem}
\end{comment}
\begin{proof}
 Let $K(\tilde C)$ denote the locally constant sheaf of rational functions on $\tilde C$. A section of $\Omega_{\tilde C}^\vee\otimes K(\tilde C)$ is contained in $\Omega_C^\vee$ if and only if its image under the push-forward map:
 \[\nu_*\colon\nu_*\hhom(\Omega_{\tilde C},K(\tilde C))\to \hhom(\Omega_C,K(\tilde C)) \]
 lies in the subspace $\hhom(\Omega_C,\OO_C)$. Since in any case $\widetilde{\mathfrak m}^4\subseteq\mathfrak m$ (see the proof of Proposition \ref{prop:classification}), vector fields vanishing up to order three certainly descend. In order to justify the remaining claims, we may work locally around the singular point in the coordinates of Section \ref{sec:sing}.
\begin{description}[leftmargin=0pt]
 \item[$(A_4):$] In the coordinates $x=t^2+ct^3,y=t^4,z=t^5$ (they are redundant, but this will be irrelevant), the section $f(t)\frac{d}{dt}\in\nu_*\Omega_{\tilde C}^\vee\otimes K(\tilde C)$ pushes forward to
 \[\nu_*\left(f(t)\frac{d}{dt}\right)=(2t+3ct^2)f(t)\frac{d}{dx}+4t^3f(t)\frac{d}{dy}+5t^4f(t)\frac{d}{dz},\]
 from which, writing $f(t)=f_0+f_1t+f_2t^2+O(t^3)$, we see that
 \[(2t+3ct^2)f(t),4t^3f(t),5t^4f(t)\in\hat\OO_{C,p}\Leftrightarrow f_0=0,cf_1+2f_2=0.\]
 
 \item[$(A_5):$] In the coordinates $x=t_1\oplus at_2+bt_2^2,y=t_1^3$ (we have $a\neq 0$), the section $f_1(t_1)\frac{d}{dt_1}\oplus f_2(t_2)\frac{d}{dt_2}$ pushes forward to
 \[\nu_*\left(f_1(t_1)\frac{d}{dt_1}\oplus f_2(t_2)\frac{d}{dt_2}\right)=\left(f_1(t_1)\oplus (a+2bt_2) f_2(t_2)\right)\frac{d}{dx}+3t_1^2f_1(t_1)\frac{d}{dy},\]
 from which, writing $f_i(t_i)=f_{i0}+f_{i1}t_i+f_{i2}t_i^2+O(t_i^3),i=1,2$, we see that
 \[f_1(t_1)\oplus (a+2bt_2)f_2(t_2),3t_1^2f_1(t_1)\in\hat\OO_{C,p} \Leftrightarrow \begin{cases} f_{10}=f_{20}=0,\\ f_{11}=f_{21},\\ 2bf_{21}+af_{22}=a^2f_{12}.\end{cases}\]
 
  \item[$(I_{m\geq 2}):$] In the coordinates of \eqref{coordIII-cs},
 \[\nu_*\left(\sum_{i=1}^m f_i(t_i)\frac{d}{dt_i}\right)=\sum_{i=1}^{m-1}\left(f_i(t_i)\oplus3\gamma_{i,m}t_m^2f_m(t_m)\right)\frac{d}{dx_i}+\\(2t_m+3\gamma_{m,m}t_m^2)f_m(t_m)\frac{d}{dx_m},\]
 hence we deduce that
 \[\nu_*\left(\sum_{i=1}^m f_i(t_i)\frac{d}{dt_i}\right)\in\Omega_C^\vee\otimes\hat\OO_{C,p}\Leftrightarrow \begin{cases} f_{i0}=0 & i=1,\ldots,m,\\ f_{i1}=3f_{m1}, & i=1,\ldots,m-1,\\ 3\gamma_{m,m}f_{m1}+2f_{m2}=0.\end{cases}\]
 
 \item[$(I\!I_{m\geq 3}):$] In the coordinates of \eqref{coordII-cs},
 \begin{multline*}\nu_*\left(\sum_{i=1}^m f_i(t_i)\frac{d}{dt_i}\right)=\left(f_1(t_1)\oplus(\alpha_{1,m}+2\beta_{1,m}t_m) f_m(t_m)\right)\frac{d}{dx_1}+\\
 \sum_{i=2}^m\left(f_i(t_i)\oplus2\beta_{i,m}t_mf_m(t_m)\right)\frac{d}{dx_i},\end{multline*}
 hence we deduce that
 \[\nu_*\left(\sum_{i=1}^m f_i(t_i)\frac{d}{dt_i}\right)\in\Omega_C^\vee\otimes\hat\OO_{C,p}\Leftrightarrow \begin{cases} f_{i0}=0 & i=1,\ldots,m,\\ 2f_{11}=f_{i1}=2f_{m1}, &  i=2,\ldots,m-1,\\ \beta_{1,m}f_{m1}+\alpha_{1,m}f_{m2}=\alpha_{1,m}^2f_{12}.\end{cases}\]
 
\end{description}
\end{proof}

As far as proper curves are concerned, there is an important distinction to make when all the branches of the genus two singularity are rational and $1$-marked, as can be seen from the appearance of the parameters $\beta_{1,m}$ and $\gamma_{m,m}$ in the previous lemma.

\begin{definition}\label{def:atom}
 The \emph{atom} of type $I_m$ (the name is borrowed from \cite{AFS1}) is obtained by gluing the subalgebra of $\k[t_1]\oplus\ldots\oplus\k[t_m]$ generated by $x_1,\ldots,x_m$ as in \eqref{coordIII} with $m$ copies of $(\k[s],(s))$ under the identification $s_i=t_i^{-1}$.  Consider it as an $m$-marked curve by marking the points with $s_i=0$ for $i=1,\ldots,m$. The multiplicative group $\Gm$ acts on the  atom by $\lambda.t_i=\lambda^3 t_i$ for $i=1,\ldots,m-1$ and $\lambda.t_i=\lambda t_i$ for $i=m$.
 
 Similarly, the atom of type $I\!I_m$ is obtained by gluing the subalgebra of $\k[t_1]\oplus\ldots\oplus\k[t_m]$ generated by $x_1,\ldots,x_{m-1}$ (and $y$) as in \eqref{coordII} (and following lines) with $m$ copies of $(\k[s],(s))$ under the identification $s_i=t_i^{-1}$. Consider it as an $m$-marked curve by marking the points with $s_i=0$ for $i=1,\ldots,m$. There is a $\Gm$-action on the type $I\!I$ atom by $\lambda.t_i=\lambda t_i$ for $i=1,m$ and $\lambda.t_i=\lambda^2 t_i$ for $i=2,\ldots,m-1$.
 
 A \emph{non-atom} of type $I$ (resp. $I\!I$) is the proper $m$-marked curve of genus two obtained by the same procedure as above when starting from an algebra of the form \eqref{coordIII-cs} with $\gamma_{m,m}\neq 0$ (resp. \eqref{coordII-cs} with $\beta_{1,m}\neq 0$). The $\Gm^m$ action on the pointed normalisation implies that the non-atoms of type $I$ (resp. $I\!I$) are all isomorphic to one another, independently of the choice of $\gamma_{i,m}\in\k^\times$ (resp. $\alpha_{1,m},\beta_{i,m}\in\k^\times$), $i=1,\ldots,m$.
\end{definition}

Finally, we describe explicit conditions for a reduced, proper, Gorenstein curve of genus two to have a finite automorphism group. These conditions are entirely combinatorial as long as there is no subcurve with a type $I_m$ (resp. $I\!I_m$) singularity and exactly $m$ special points. 

Recall Smyth's description of genus one curves with no infinitesimal automorphisms \cite[Proposition 2.3, Corollary 2.4]{SMY1}.

\begin{definition}
 Let $(C,p_1,\ldots,p_n)$ be a reduced pointed curve. A connected subcurve $D\subseteq C$ is said to be \emph{nodally attached} if $D\cap\overline{C\setminus D}$ consists of nodes only.
  We say that $C$ is \emph{residually DM} (rDM) if every nodal and nodally attached subcurve $D$ of $C$, marked by $\{p_i\in D\}\cup (D\cap\overline{C\setminus D}$), is Deligne-Mumford stable. As usual, by \emph{special points} we mean markings and nodes.
\end{definition}

\begin{cor}\label{cor:explicitnoaut}
 Let $(C,p_1\ldots,p_n)$ be a Gorenstein pointed curve of arithmetic genus two over a field of characteristic $\neq2,3$. $H^0(C,\Omega_C^\vee(-\sum_{i=1}^n p_i))=0$ is equivalent to either of the following:
 \begin{enumerate}[leftmargin=.6cm]
  \item $C$ has a singularity of type $I_{m\geq 1}$: either all branches contain exactly one special point and $C$ is the non-atom; or each of its axes contains at least one special point, and at least one branch has at least two. Furthermore $C$ is rDM.
  \item $C$ has a singularity of type $I\!I_{m\geq 2}$: either all branches contain exactly one special point and $C$ is the non-atom; or at least one of its twin branches contains a special point, each of its axes contains at least one, and at least one branch has at least two. Furthermore $C$ is rDM.
  \item $C$ has two elliptic $m$-fold points: each of their branches contains at least one special point or is shared, and at least one branch for each singular point contains at least one extra special point. Furthermore $C$ is rDM.
  \item $C$ has one elliptic $m$-fold point: either one of its branches is a genus one curve, and every other branch contains at least one special point; or all branches contain at least one special point, and either two of its branches coincide, or at least one branch has at least two special points. Furthermore $C$ is rDM.
  \item $C$ contains only nodes and is Deligne-Mumford stable.
 \end{enumerate}
\end{cor}

\begin{definition}\label{def:dangling}
 A curve with a singularity of type $I\!I$ such that one of the special branches contains no special points is called \emph{dangling} (see \cite[\S 2.1]{AFSGm}).
\end{definition}

\section{Admissible covers and semistable tails}\label{sec:sstails}

Given a family of prestable (pointed) curves of genus two over the spectrum of a discrete valuation ring $\mathcal C\to\dvr$, with smooth generic fibre $\mathcal C_{\eta}$ and regular total space, we classify the subcurves of the central fibre $\mathcal C_{0}$ that can be contracted to yield a Gorenstein singularity of genus two. 

In the genus one case, Smyth answered the analogous question by identifying the class of \emph{balanced} subcurves \cite[Definition 2.11]{SMY1}: subcurves of arithmetic genus one, such that, when breaking them into a \emph{core} (minimal subcurve of genus one, not containing any separating node) and a number of rational trees (with root corresponding to the component adjacent to the core, and leaves corresponding to the components adjacent to the portion of $\mathcal C_0$ that is not contracted), the distance between any leaf and the root of any such tree is constant, not depending on the tree either.

In the case at hand, the answer turns out to be slightly more complicated: first, the special branch(es) of a type $I$ (resp. $I\!I$) singularity are connected through rational chains to a Weierstrass (resp. two conjugate) point(s) of the core. Second, the lengths of the rational trees may vary according to where their attaching points lie, but the special chains are always the shortest, and, together with the configuration of the attaching points on the core, they determine the length of any other chain.

\subsection{A quick recap on admissible covers}\label{rmk:Wandconj}
 While there are no special points on a smooth curve of genus zero or one, the simplest instance of Brill-Noether theory involves smooth curves of genus two. Every such $C$ is \emph{hyperelliptic}: it admits a unique (up to reparametrisation) two-fold cover $\phi\colon C\to\PP^1$, induced by the complete canonical linear system, i.e. $\lvert K_C\rvert$ is the unique $\mathfrak g^1_2$ on $C$; said otherwise, there is a unique element $\sigma\in\Aut(C)$, called the \emph{hyperelliptic involution}, such that $C/\langle\sigma\rangle\simeq\PP^1$. A point $x\in C$ is called \emph{Weierstrass} if it is a ramification point for $\phi$ (or, equivalently, a fixed point for $\sigma$); from the Riemann-Hurwitz formula it follows that there are six Weierstrass points on every smooth curve of genus two. Two points $x_1,x_2$ are said to be conjugate (write $x_2=\bar x_1$) if there exists a point $z\in\PP^1$ such that $\phi^{-1}(z)=\{x_1,x_2\}$ (or, equivalently, $\sigma(x_1)=x_2$). These notions may be extended to nodal curves by declaring $(C,x)$ to be Weierstrass if its stabilisation lies in the closure of
 \[\mathcal W=\{(C,x)|\ C\text{ smooth and } x \text{ Weierstrass}\}\subseteq\oM_{2,1},\]
 and similarly for conjugate points. We then need to study the limiting behaviour of Weierstrass points when a smooth curve degenerates to a nodal one. This is a difficult problem when it comes to higher genus curves; it has received considerable attention since the '70s, in work of E. Arbarello, D. Eisenbud, J. Harris, and many others. In our case it boils down to understanding admissible covers \cite{HarrisMumford} of degree two with a branch divisor of degree six.

\begin{definition}\label{def:admissiblecover}
 A family of \emph{pointed hyperelliptic admissible covers} over $S$ is a finite morphism $\psi\colon (C,D_R,p_1,\bar p_1,\ldots,p_n,\bar p_n)\to (T,D_B,\psi(p_1),\ldots,\psi(p_n))$ over $S$ such that:
 \begin{enumerate}
  \item $(C,D_R,\mathbf{p},\bar{\mathbf{p}})$ and $(T,D_B,\psi(\mathbf{p}))$ are prestable curves, with unordered smooth disjoint multisections $D_R$ and $D_B$ of length $2g+2$, and ordered smooth disjoint multisections $\mathbf{p},\bar{\mathbf{p}},\psi(\mathbf{p})$;
  \item $C$ has arithmetic genus $g$, and $(T,D_B,\psi(\mathbf{p}))$ is a stable rational tree;
  \item $\psi$ is a double cover on an open $U\subseteq T$ dense over $S$;
  \item\label{point:logetale} $\psi$ is \'etale on $C^\text{sm}\setminus D_R$, and $\psi(p_i)=\psi(\bar p_i)$ for $i=1,\ldots,n$;
  \item $\psi$ maps $D_R$ to $D_B$ with simple ramification, and it maps nodes of $C$ to nodes of $T$ so that in local coordinates:
  \[ \psi^\#\colon \OO_S[u,v]/(uv-s)\to\OO_S[x,y]/(xy-t)\]
  maps $u\mapsto x^i,v\mapsto y^i, s\mapsto t^i$ for $i=1$ or $2$.
 \end{enumerate}
\end{definition}

 \begin{thm}[{\cite{HarrisMumford,Mochizuki}}]
  The moduli stack of pointed hyperelliptic admissible covers $\overline{\mathcal H}_{g,n}$ is a proper and smooth Deligne-Mumford stack with normal crossing boundary and forgetful morphisms $s\colon \overline{\mathcal H}_{g,n}\to\oM_{g,n}$ and  $t\colon \overline{\mathcal H}_{g,n}\to\oM_{0,2g+2|n}$.
 \end{thm}

 $\overline{\mathcal H}_{g,n}$ provides a nice compactification of the locus of hyperelliptic curves in $\oM_{g,n}$. Besides the original sources,  we have benefited from the exposition in \cite[Appendix 2]{Diaz}, \cite[Proposition (3.0.6)]{Cukierman}, and \cite[Theorem 5.45]{HM}.
 
 We extract some information on the Weierstrass and conjugate loci in genus two: up to the involution action, the Weierstrass locus is isomorphic to $\oM_{0,6}/\mathfrak S_5$, and the conjugate locus is isomorphic to  $\oM_{0,7}/\mathfrak S_6$. We remark that $(C,x)$ being Weierstrass is an intrinsic notion if $C$ is of compact type, but it may depend on the smoothing otherwise. Indeed, if we let $\overline{\mathcal W}_{2,n}$ (resp. $\overline{\mathcal K}_{2,n}$) denote the space of hyperelliptic admissible covers of genus two with a marked Weierstrass (resp. two conjugate) point(s) and $n$ further markings, forgetting the Weierstrass (resp. one of the two conjugate) point(s) is \emph{not} a finite map to $\oM_{2,n}$ (resp. $\oM_{2,n+1}$) - see below the case when $x$ belongs to a rational bridge. We spell out an explicit description of the image of $\overline{\mathcal W}_{2,n}$ and $\overline{\mathcal K}_{2,n}$.
 \begin{itemize}[leftmargin=.5cm]
  \item If $x$ belongs to a component of genus one $E$, which is attached to another component of genus one at a node $y$, then $x$ is Weierstrass iff $2x\sim 2y\in\Pic(E)$; if instead $E$ has a self-node that glues $y_1$ with $y_2$, then $x$ is Weierstrass iff $2x\sim y_1+y_2\in\Pic(E)$.
  
  If $x$ is on a rational component $R$, $x$ is Weierstrass if either $R$ is attached to a genus one curve at two distinct points; or $R$ has a self-node gluing $y_1$ and $y_2$ and is attached to a genus one tail at $y_3$, in which case we require $\phi(y_1)=\phi(y_2)$ for a double cover $\phi\colon R\to\PP^1$ ramified at $x$ and $y_3$; or $R$ has two self-nodes gluing $y_1$ with $y_2$, and $y_3$ with $y_4$, in which case we require $x$ to be a ramification point for a double cover $\phi\colon R\to\PP^1$ such that $\phi(y_1)=\phi(y_2)$ and $\phi(y_3)=\phi(y_4)$ - geometrically, if we embed $\PP^1$ as a conic $C\subseteq\PP^2$, the line through $x$ and $\overline{y_1y_2}\cap\overline{y_3y_4}$ should be tangent to $C$ at $x$. See Figure \ref{fig:adm_W}.
  \begin{center}
  \begin{figure}[!ht]
  \includegraphics[width=\textwidth]{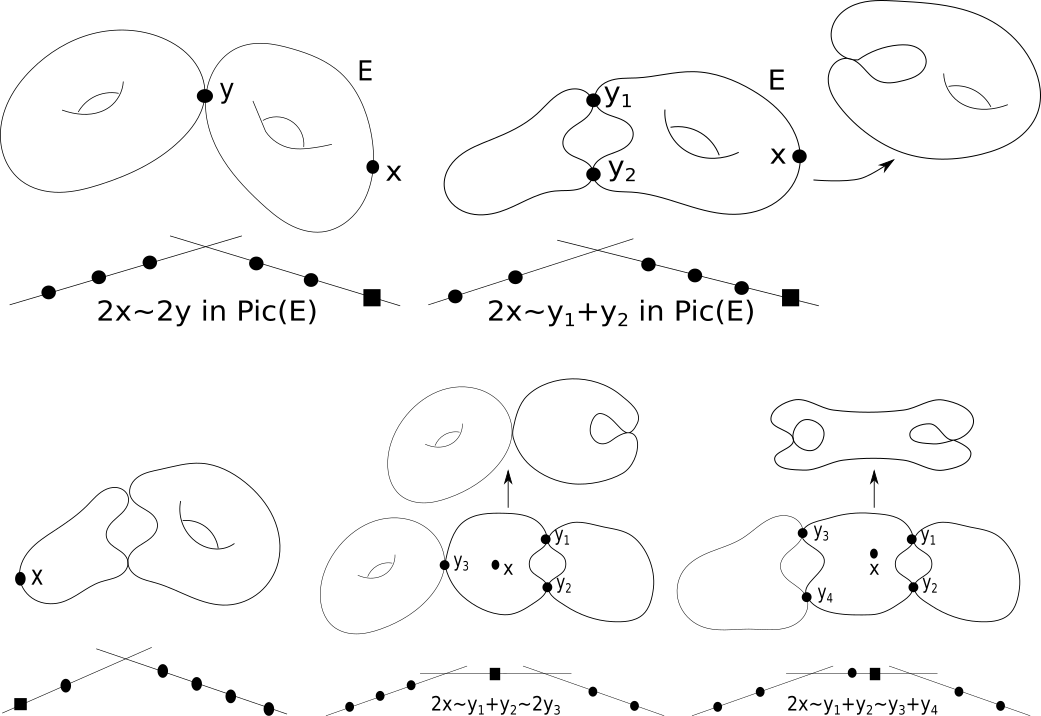}
  \caption{Admissible covers and Weierstrass points.}\label{fig:adm_W}
  \end{figure}
  \end{center}
 \item If $x_1$ and $x_2$ are conjugate, they have to map to the same component of the target of the admissible cover. We may adapt the description of the previous point by replacing every condition on $2x$ by its analogue for $x_1+x_2$ (in fact, attaching a rational component with two extra markings to a Weierstrass point always produces an element of $\overline{\mathcal K}_{2,n}$, so knowing the latter determines $\overline{\mathcal W}_{2,n}$).  There are a few more situations to take into account: %$x_1$ and $x_2$ could belong to a rational component $R$ bubbling off from a Weierstrass point of a genus two curve; or 
 $x_1$ and $x_2$ could belong to a rational component $R$ bridging between two distinct curves of genus one; or $x_1$ and $x_2$ could lie on two distinct rational components $R_1$ and $R_2$ intersecting each other at one node and meeting a curve of genus one in two distinct points (\dag); or $R_1$ and $R_2$ intersecting each other in three points. See Figure \ref{fig:adm_conj}.
  \begin{figure}[!ht]
 \includegraphics[width=\textwidth]{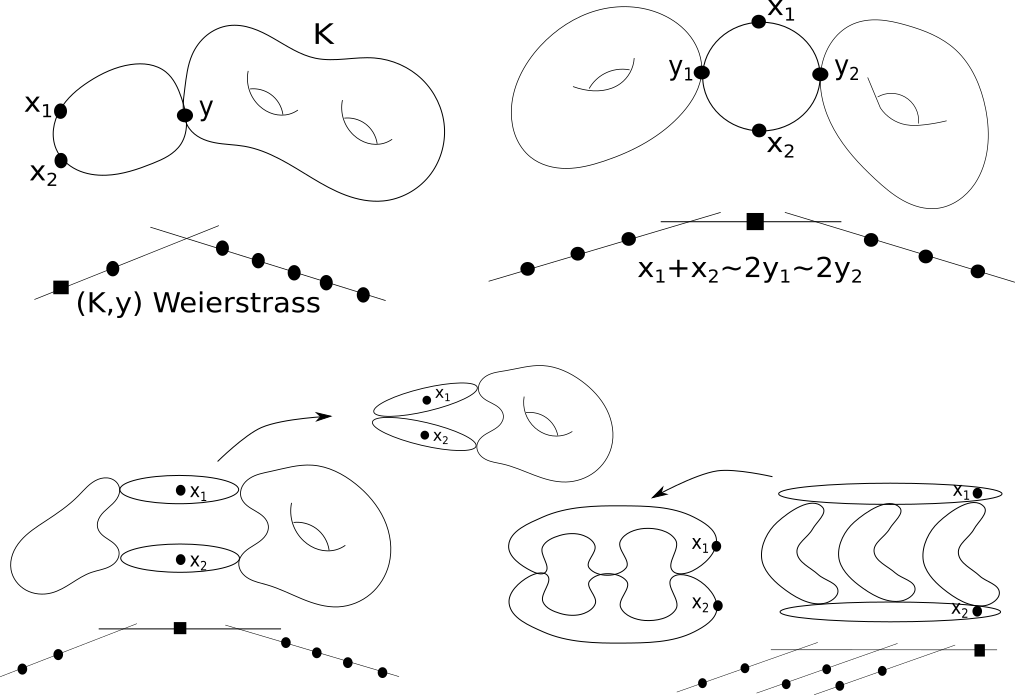}  
 \caption{Admissible covers and conjugate points.}\label{fig:adm_conj}
  \end{figure}
 \end{itemize}
 
 \begin{rem}\label{rmk:necklace_symmetry}
  In case (\dag), the singularity of the total space of a smoothing $\mathcal C\to\dvr$ at the two distinguished nodes (separating the elliptic component from the rational chain) are both $A_k$ for the same $k$, because they map to the same node of the target in the admissible cover. This consideration is stable under base change, and it therefore entails a symmetry of the rational chain in the model with regular total space.
 \end{rem}
 
 \subsection{A quick recap on logarithmically smooth curves}\label{sec:recaplog}
 Logarithmically smooth curves are prestable curves endowed with a suitable logarithmic structure \cite{KatoF}. There is a minimal \cite{Gillam} such logarithmic structure $(C, M_C^\text{min})\to(S, M_S^\text{min})$, determining all the others by pullback: this is the logarithmic structure on the moduli stack of prestable curves $\mathfrak M_{g,n}$ induced by its normal crossing boundary. More explicitly, $ M_S^\text{min}$ is a locally free logarithmic structure, with generators of the characteristic sheaf $\overline M_S^\text{min}$ corresponding to nodes of the curve. The tropicalization $\tropC$ of a logarithmically smooth curve $(C, M_C)$ over a geometric point $(S=\operatorname{Spec}(\bar\k),M_S)$ consists of its dual graph $\Gamma(C)$ metrised in $\overline M_S$ - the length of an edge is its smoothing parameter. In particular, a family of prestable curve over a trait $\dvr$ gives rise to a standard tropical curve (metrised in $\mathbb N$). After \cite{GS} and \cite{CCUW}, piecewise-linear (PL) functions on $\tropC$ with values in $\overline M_S$ correspond to sections of $\Gamma(C,\overline M_C)$. The latter determine $\OO_C^*$-torsors (and therefore line bundles) on $C$ by the short exact sequence:
 \[0\to\OO_C^*\to M_C^\text{gp}\to \overline M_C^\text{gp}\to 0.\]
 A detailed analysis of this correspondence can be found in \cite[Proposition 2.4.1]{RSPW1}. See also \cite[p.9]{Bozlee} for a description in local charts.
 
 The moduli space of pointed hyperelliptic admissible covers $\overline{\mathcal H}_{g,n}$ is also logarithmically smooth with locally free logarithmic structure induced by the normal crossing boundary \cite{Mochizuki}. Generators of $\overline M_{\bar{\mathcal H}}^\text{min}$ correspond to the nodes of the source curve $C$, with two of them being identified if the cover is a local isomorphism around them. This allows us to use the language of tropical geometry, see for instance \cite{CMR}. In particular, the tropicalisation of an admissible cover is a \emph{harmonic} map $\psi\colon \tropC\to\tropT$ satisfying the \emph{local Riemann-Hurwitz condition}. This means that every edge of $\tropT$ has either two preimages, or one with \emph{expansion factor} $2$; and that for every vertex $v$ of $\tropC$ the genus of the corresponding irreducible component satisfies: \[2g(v)-2=-4+R\]
 where $R$ denotes the number of edges of expansion factor $2$ and legs corresponding to the branch divisor (we call them $B$-legs) adjacent to $\psi(v)$.
 
 \subsection{Minimal curves}
 \begin{definition}
  A projective Gorenstein curve $C$ is \emph{minimal} if it contains no node $x$ such that the normalisation of $C$ at $x$ consists of two connected components, one of which has genus zero.
 \end{definition}

 When $C$ is nodal, minimal is equivalent to semistable (no rational tails). Compare with \cite[Definition 3.2]{Catanese} for an even stronger notion. When $C$ has arithmetic genus one, this is the same as saying that $C$ contains no separating nodes. Recall \cite[Lemma 3.3]{SMY1}.

\begin{lem}\label{lem:min1}
 A minimal Gorenstein curve $E$ of arithmetic genus one can be: a smooth elliptic curve; a ring of $r\geq 1$ copies of $\PP^1$; or an elliptic $m$-fold point whose normalisation is the disjoint union of $m$ copies of $\PP^1$. In any case $\omega_E\simeq\OO_E$.
\end{lem}

We provide a similar description of minimal curves of genus two; the proof is left to the reader. By a semistable rational chain of length $k$ we mean the nodal union of $k$ copies of  $(\PP^1,0,\infty)$, so that $\infty_i$ is identified with $0_{i+1}$ for $i=1,\ldots,k-1$; if $k=0$, we mean a point.
\begin{lem}\label{lem:min2}
 A minimal Gorenstein curve of genus two can be either of the following (Figure \ref{fig:minimalcurves}):
 \begin{enumerate}[label=(\alph*)]
  \item a smooth curve of genus two;
  \item the union of two minimal Gorenstein curves of genus one, $E_1$ and $E_2$, nodally separated by a semistable rational chain of length $k\geq 0$;
  \item the nodal union of a minimal Gorenstein curve of genus one $E$ and a semistable rational rational chain of length $k\geq 0$;
  \item the union of two copies of $(\PP^1,0,1,\infty)$ with three semistable rational chains $R_0, R_1, R_\infty$ (of length $k_0,k_1,k_\infty\geq 0$) joining the homonymous points;
  \item\label{case:twice1} an elliptic $m$-fold point whose pointed normalisation is the disjoint union of either $m-2$ copies of $(\PP^1,0)$ and a semistable rational chain $R$ of length $k\geq 1$, or $m-1$ copies of $(\PP^1,0)$ and a $1$-pointed minimal Gorenstein curve of genus one (if the latter is not irreducible and $m\neq 1$, there are two genus one subcurves sharing a rational branch);
  \item\label{case:2} or a singularity of genus two with $m$-branches, whose normalisation is the disjoint union of $m$ copies of $\PP^1$.
 \end{enumerate}
\end{lem}

\begin{center}
 \begin{figure}
  \includegraphics[width=\textwidth]{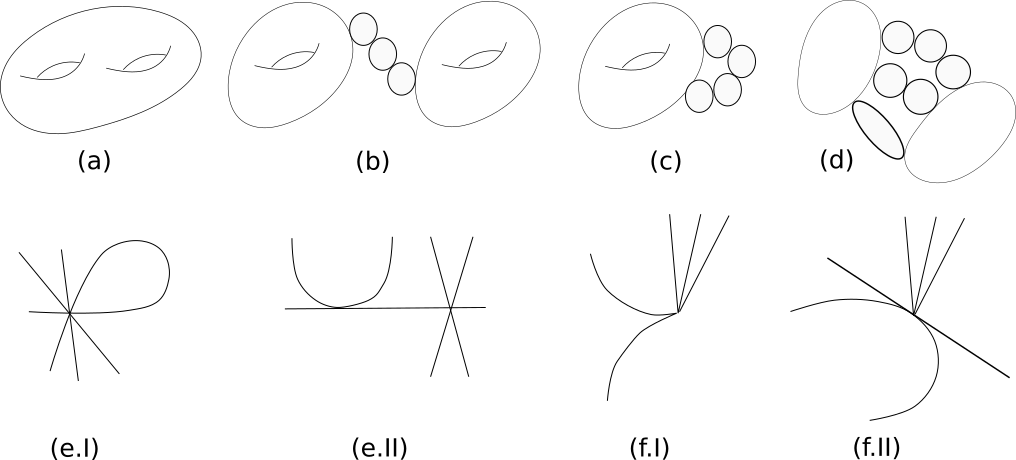}
  \caption{Minimal curves of genus two.}
\label{fig:minimalcurves}
 \end{figure}
\end{center}

\begin{rem}\label{rem:special}
In case $Z$ is a minimal curve of genus two of type \ref{case:twice1} or \ref{case:2} above, there are \emph{special} components supporting the degree of $\omega_Z$. For case \ref{case:2}, see Definition \ref{def:special_branches} and Corollary \ref{cor:dualising_line_bundle}. For case \ref{case:twice1}, the special component is either the genus one branch, or the rational component that contributes two branches to the singularity (recall that the restriction of the dualising sheaf to a component introduces a twist by the conductor ideal, see \cite[Proposition 1.2]{Catanese}).
\end{rem}

\subsection{Semistable tails}
 Let $\overline{\mathcal C}_0$ be a minimal curve with a genus two singularity of type $I$ (resp. $I\!I$), and let $\overline{\mathcal C}$ be a one-parameter smoothing over a trait $\dvr$, with closed point $0$ and generic point $\eta$. Let $\mathcal P$ denote $\PP(\pi_*\omega_{\overline{\mathcal C}/\dvr})$, which is a $\PP^1$-bundle over $\dvr$. It follows from an easy calculation (or from \cite[Theorem D]{Catanese}) that the canonical series is basepoint-free, and so there is a morphism:
 \bcd
 \overline{\mathcal C}\ar[dr, "\bar\pi" below left]\ar[rr] & & \mathcal P \ar[dl] \\
 & \dvr &
 \ecd
 such that, in the central fibre, it restricts to a double cover on the special branch (resp. an isomorphism on each of the special branches) and it contracts the axes. The geometric general fibre is the hyperelliptic cover $\overline{\mathcal C}_{\bar\eta}\to\PP^1_{\bar\eta}$, endowing $\PP^1_{\bar\eta}$ with a simple branch divisor $B_{\bar\eta}$ of length $6$. Possibly after passing to a finite cover of $\dvr$, $B_{\eta}$ itself splits into the union of six disjoint sections, and we can take the stable model $(\mathcal T,B)$ of $(\PP^1_\eta,B_\eta)$, together with its associated double cover $\mathcal C$. We thus have a diagram:
 \bcd
 \mathcal C\ar[r,"\psi"]\ar[d,"\phi"] & \mathcal T\ar[d]\\
 \overline{\mathcal C} \ar[r] & \mathcal P
 \ecd
 over $\dvr$ (by a slight abuse of notation), where the upper row is a family of admissible covers. 
 
 The line bundle $\OO_{\mathcal P}(1)$ pulls back to $\omega_{\bar \pi}$ on $\overline{\mathcal C}$. Its pullback $\OO_{\mathcal T}(1)$ on $\mathcal T$ has degree $1$ on exactly one component of the tree. Pulling back further to $\mathcal C$, we gather the following information:
 \begin{enumerate}[label=(\alph*)]
  \item $\phi^*\omega_{\bar\pi}=\omega_{\pi}(\pazocal Z)$ for a vertical divisor $\pazocal Z$ supported on the exceptional locus $\operatorname{Exc}(\phi)=: Z$.
  \item $\psi^*\OO_{\mathcal T}(1)=\OO_{\mathcal C}(q+\bar q)$ for a choice of two conjugate points of $\mathcal C$ lying over the same point of $\mathcal T$, belonging to the component on which $\OO_{\mathcal T}(1)$ is ample.
  \item\label{pt:compatibleZ} $\pazocal Z$ is the pullback of a vertical divisor on $\mathcal T$.
 \end{enumerate}
 
 This description leads to the following simple observations:
 
 \begin{enumerate}[label=(\Roman*)]
  \item If $\overline{\mathcal C}_0$ has a type $I$ singularity, the branch of ${\mathcal C}_0$ corresponding to the singular branch of $\overline{\mathcal C}_0$ is attached to a Weierstrass point of $\pazocal Z$ with respect to $\psi$.
  \item If $\overline{\mathcal C}_0$ has a type $I\!I$ singularity, the branches of ${\mathcal C}_0$ corresponding to the twin branches of $\overline{\mathcal C}_0$ are attached to two conjugate points of $\pazocal Z$ with respect to $\psi$.
 \end{enumerate}

 Moreover, the distance of the special branch(es) from the core is always less than that of the axes; the ratio is roughly $1:3$ in case $I$, and $1:2$ in case $I\!I$, but, more precisely, this depends on the relative position of the attaching points of the chains in the dual graph of the core. An elegant treatment uses the language of tropical geometry.

 We consider the tropicalization $\tropC\to\tropT$ of $\psi$, as in Section \ref{sec:recaplog}. After further base-change and normalised blow-ups, we can assume that $\mathcal C$ has regular total space; this only affects $\tropC$ by subdividing edges, not changing their lengths. Now $\tropC$ is nothing but the dual graph of the special fiber $\mathcal C_0$, with edges of length $1$.
 
 The vertical divisor $\pazocal Z$ can be represented by a piecewise-linear function on $\tropC$ with integral slope along the edges; moreover, observation \ref{pt:compatibleZ} above shows that $\lambda$ is pulled back from a piecewise-linear function $\lambda_T$ on $\tropT$ - for this to be true we have to allow half-integral slopes along the edges. Finding $\lambda$ becomes a simple matter of degree-matching on the tree $\tropT$; this shows existence and uniqueness (up to global translation).
 
 Recall that the canonical divisor $K_{\tropC}$ has the following multiplicity on a vertex $v$ of $\tropC$:
 \begin{equation}\label{eq:canonical_degree}
 2g(v)-2+\operatorname{val}(v),                                                                                                                                                                                                                                                                                                                                                                                                                                                                                                                                                                                                                                                                                                                                                                                                                                                                                                                                                                                                                                                             \end{equation}
where $g\colon V(\tropC)\to\mathbb Z$ is the genus assignment, and $\operatorname{val}(v)$ is the number of bounded edges adjacent to $v$; \eqref{eq:canonical_degree} is also the degree of $\omega_\pi$ when restricted to the component of $\mathcal C_0$ corresponding to $v$.

Notice that $\tropT$ is decorated with six unlabelled $B$-legs corresponding to the branch divisor $B$. It follows from the local Riemann-Hurwitz condition of Section \ref{sec:recaplog} that the divisor $\OO_{\tropT}(1)$ on $\tropT$ pulling back to $K_{\tropC}$ has the following multiplicity at a vertex $v^\prime$ of $\tropT$:
 \[ \operatorname{val}(v^\prime)-2+\frac{1}{2}\#\{B\text{-legs adjacent to } v^\prime\} \]
(notice that pulling back doubles the multiplicity of points with a single preimage). Therefore, the equation that we have to solve in order to find $\lambda_T$ is:
 \begin{equation}\label{eqn:Dv}
  \operatorname{val}(v^\prime)-2+\frac{1}{2}\#\{B\text{-leg adjacent to } v^\prime\}+\sum_{e\text{ bounded edge adjacent to } v^\prime}s(\lambda_T,e)
 \end{equation}
 equals $1$ on the vertex of $\tropT$ corresponding to the special branch, and $0$ otherwise. Here $s(\lambda_T,e)$ denotes the outgoing slope of $\lambda_T$ along the edge $e$.

  For the benefit of the reader, we include Figure \ref{fig:adm_fun_smooth_core} to illustrate the shape of $\lambda$ in the simplest possible case, namely when the core is smooth. The blue numbers represent the slope of $\lambda$ along the corresponding edges. Figure \ref{fig:adm_fun_reducible_core} exhibits how the distance of the axes from the core can vary when the latter becomes more degenerate.

  \begin{figure}[htb]
  \centering
  \includegraphics[width=.8\textwidth]{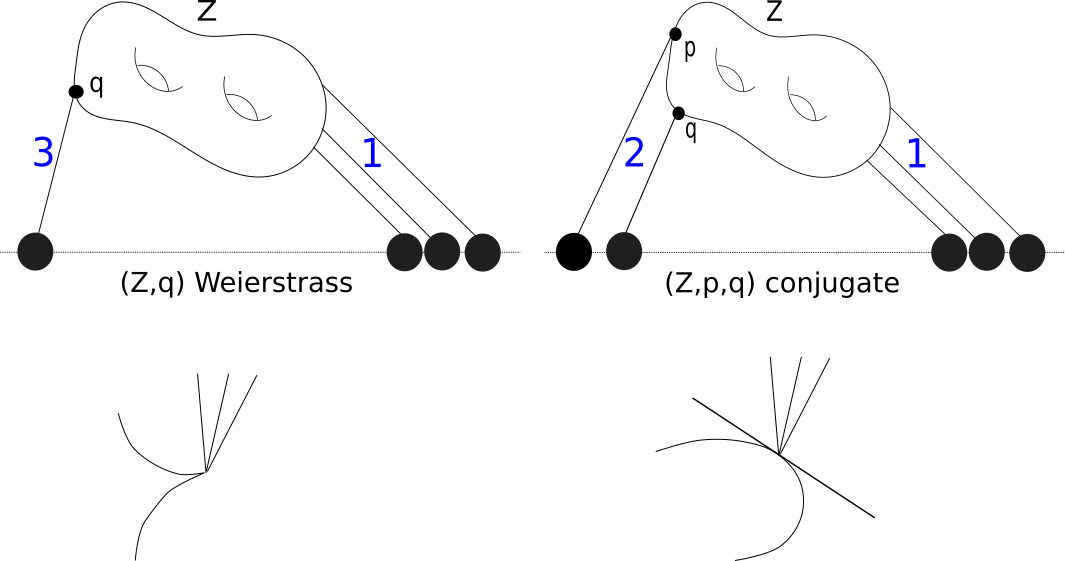}
  \caption{Semistable tail of a type $I_4$ (left), resp. type $I\!I_5$ (right), singularity, generic case: the core is smooth, the singular branch is attached to a Weierstrass point (resp. the twin branches are attached to conjugate points), the other branches are attached to distinct points, and the corresponding edge-length is three (resp. two) times longer than the special one.}
  \label{fig:adm_fun_smooth_core}
 \end{figure}

  \begin{figure}[htb]
  \centering
  \includegraphics[width=.7\textwidth]{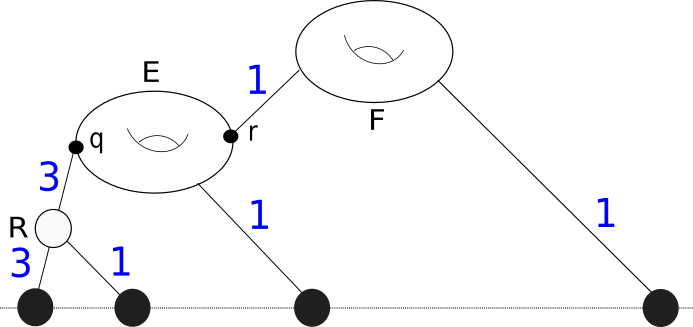}
  \caption{A more degenerate semistable tail of a type $I_4$ singularity. Here $Z$ consists of $R$, $E$, and $F$ together. $(Z,q)$ is Weierstrass in the sense that $2q\sim 2r\in\operatorname{Pic}(E)$.}
  \label{fig:adm_fun_reducible_core}
 \end{figure}

 Two important observations allow us to write down $\lambda$ explicitly in all possible situations:
 \begin{enumerate}
  \item The balancing equation \eqref{eqn:Dv} is unaffected by \emph{tropical modifications}, i.e. growing a tree on which $\lambda$ has constant slope $1$.
  \item The balancing equation \eqref{eqn:Dv} is stable under edge contraction.
 \end{enumerate}
 It follows that it is enough to study the case that the core consists of a configuration of rational curves; there are only two stable such configurations, named \emph{dumbbell} and \emph{theta}. Figure \ref{fig:admredcores} (from \cite{BC}) illustrates the situation: we draw both the source (above) and the target (below) of the tropical admissible cover; the blue numbers on the latter represent the slope of $\lambda_T$ - that of $\lambda$ can be recovered by multiplying with the expansion factor of $\psi$; and the red vertices correspond to the special branches. The vertices corresponding to the axes of the genus two singularity do not appear in the picture: they lie at the same height as the red vertices, on an arbitrary configuration of trees emanating from the core, along which $\lambda$ has slope $1$.

\begin{figure}
	\centering
	%\begin{minipage}[t]{0.5\textwidth}
		%\centering
		\begin{tikzpicture} [scale=.6]
		\tikzstyle{every node}=[font=\normalsize]
		\tikzset{arrow/.style={latex-latex}}
		\def\S{2.825cm} % Seitenlaenge des Quadrats 

		%coordinates source
		
		\coordinate (A) at (0,6,0);
		  \coordinate (B) at (2,6,0);
		 \coordinate (C) at (3.5,5.5,0);
		 
		  \coordinate (CM) at (3.5,5,0);;%need for case 5
		 	\coordinate (AM) at (5,6,0);%need for case 5
		  \coordinate (BM) at (7,6,0);;%need for case 5

		 \coordinate (D) at  (5,4.5,0); 
		 \coordinate (Dd) at (5.5,5.5,0); %need for case 4
		\coordinate (D1) at  (5,4,-.5);%need for case 3
		\coordinate (D2) at  (5,4,0);%need for case 3
		
		\coordinate (Ew) at (6.5,5,7);
		\coordinate (Ee) at (6.5,3.8,0);%need for case 3
		\coordinate (E) at (6.5,2,0); 
         \coordinate (E1) at (6.5,3,-.5); %need for NW
         \coordinate (E2) at ( 6.5,3,0); %need for NW
      %coordinate target
      	\coordinate (At) at (0,1,0);
		  \coordinate (Bt) at (2,1,0);
		 \coordinate (Ct) at (3.5,1,0);
		\coordinate (Dt) at  (5,1,0);
		\coordinate (Et) at (6.5,1,0); 
		\coordinate (Ddt) at (5.5,1,0); %need for 4
		\coordinate (Ewt) at (6.5,1,7);); %need for case 4
		\coordinate (BMt) at (7,1,0);;%need for case 5
      %coordinate markings

		\coordinate (P1) at (-.3,1,1.3);
		\coordinate (P2) at (.3,1,1.7);
		\coordinate (P3) at (1.8,1,1.3);
		\coordinate (P4) at (3.3,1,1.3);
		\coordinate (P5) at (4.8,1,1.3);
		\coordinate (P6) at (6.3,1,1.3);
		\coordinate (Pp6) at (6.8,1,1.7); %need in 3
		\coordinate (Q6) at (5.1,1,1.7); %need later
        \coordinate (T6) at (6.5,1,6);
        
        \coordinate (QM6) at (6.7,1,1.3); %need in 5
        \coordinate (TM6) at (7.2,1,1.6); %need in 5
      %%%%%%%%%%%%%%%%%%%%%%%%%%%%%%%%%%%%%%%%%%%%
      
      %dumbbell 2
      
      % arrow source
      \draw (A) to [out=130,in=50] (B);
      \draw (A) to [out=-130,in=-50] (B);
      \draw[->-] (C)--(B);
      \draw[->-] (D) to [out=70,in=70](C);
      \draw[->-] (D) to [out=-110,in=-110](C);
       \draw[->-] (E)--(D);
         
        %arrow target
        \draw[->-] (Et)--(Dt);
        \draw[->-] (Dt)--(Ct);
        \draw[->-] (Ct)--(Bt);
        \draw (Bt)--(At);

        % Branch points
		\draw[blue] (At)--(P1);
		\draw[blue] (At)--(P2);
		\draw[blue] (Bt)--(P3);
		\draw[blue] (Ct)--(P4);
		\draw[blue] (Dt)--(P5);
		\draw[blue] (Et)--(P6);
		
		%slopes
\node at (1,1,0) [above] {\small{0}};
\node at (2.75,1,0) [above] {\small{$\frac{1}{2}$}};
\node at (4.25,1,0) [above] {\small{$1$}};
\node at (5.75,1,0) [above] {\small{$\frac{3}{2}$}};

       %admissible cover
       
       \draw[dashed] (A)--(At);
        \draw[dashed] (B)--(Bt);
       \draw[dashed] (C)--(Ct);
       \draw[dashed] (D)--(Dt);
        \draw[dashed] (E)--(Et);

      \foreach \x in {A,B,C,D,At,Bt,Ct,Dt,Et}
   \fill (\x) circle (3pt);
     	\fill[red] (E) circle (3pt);
     	
%%%%%%%%%%%%%%%%%%%%%%%%%%%%%%%%%%%%%%%%%%%%%%%%%%%%%%

%dumbbell 1

\begin{scope}[every coordinate/.style={shift={(0,9,0)}}]

% arrow source
       \draw ([c]A) to [out=130,in=50] ([c]B);
      \draw ([c]A) to [out=-130,in=-50] ([c]B);
       \draw[->-] ([c]C)--([c]B);
      \draw[->-] ([c]D1) to [out=90,in=0]([c]C);
      \draw[->-] ([c]D2) to [out=-180,in=-90]([c]C);
       \draw[->-] ([c]E1)--([c]D1);
       \draw[->-] ([c]E2)--([c]D2);
       \draw ([c]D1) -- ([c]D2);
         
        %arrow target
        \draw[->-] ([c]Et)--node[above] {\small{$2$}}([c]Dt);
        \draw[->-] ([c]Dt)--node[above] {\small{$1$}}([c]Ct);
        \draw[->-] ([c]Ct)--node[above] {\small{$\frac{1}{2}$}}([c]Bt);
        \draw ([c]Bt)--node[above] {\small{0}}([c]At);

        % Branch points
		\draw[blue] ([c]At)--([c]P1);
		\draw[blue] ([c]At)--([c]P2);
		\draw[blue] ([c]Bt)--([c]P3);
		\draw[blue] ([c]Ct)--([c]P4);
		\draw[blue] ([c]Dt)--([c]P5);
		\draw[blue] ([c]Dt)--([c]Q6);

       %admissible cover
       
       \draw[dashed] ([c]A)--([c]At);
        \draw[dashed] ([c]B)--([c]Bt);
       \draw[dashed] ([c]C)--([c]Ct);
       \draw[dashed] ([c]D1)--([c]Dt);
       \draw[dashed] ([c]D2)--([c]Dt);
        \draw[dashed] ([c]E1)--([c]Et);
        \draw[dashed] ([c]E2)--([c]Et);

      \foreach \x in {A,B,C,D1,D2,At,Bt,Ct,Dt,Et}
   \fill ([c]\x) circle (3pt);
     	\fill[red] ([c]E1) circle (3pt);
     		\fill[red] ([c]E2) circle (3pt);
\end{scope}

 %%%%%%%%%%%%%%%%%%%%%%%%%%%%%%%%%%%%%%%%%%%%%%%%%%
 
 % dumbbell 3 (last)
 
 \begin{scope}[every coordinate/.style={shift={(0,-9,0)}}]

% arrow source
       \draw ([c]A) to [out=130,in=50] ([c]B);
      \draw ([c]A) to [out=-130,in=-50] ([c]B);
       \draw[->-] ([c]CM)--([c]B);
       \draw[->-] ([c]Ew)--([c]CM);
        \draw[->-] ([c]CM)--([c]AM);
      \draw ([c]AM) to [out=130,in=50] ([c]BM);
      \draw ([c]AM) to [out=-130,in=-50] ([c]BM);
      
      %arrow target
        \draw ([c]At)--node[above] {\small{$0$}}([c]Bt);
        \draw[->-] ([c]Ct)--node[above] {\small{$\frac{1}{2}$}}([c]Dt);
        \draw[->-] ([c]Ewt)--node[left]{\small$\frac{3}{2}$}([c]Ct);
        \draw[->-] ([c]Ct)--node[above] {\small{$\frac{1}{2}$}}([c]Bt);
        \draw ([c]Dt)--node[above] {\small{$0$}}([c]BMt);

         % Branch points
		\draw[blue] ([c]At)--([c]P1);
		\draw[blue] ([c]At)--([c]P2);
		\draw[blue] ([c]Bt)--([c]P3);
		\draw[blue] ([c]Dt)--([c]P5);
		\draw[blue] ([c]BMt)--([c]QM6);
		\draw[blue] ([c]BMt)--([c]TM6);

      %admissible cover
       
       \draw[dashed] ([c]A)--([c]At);
        \draw[dashed] ([c]B)--([c]Bt);
       \draw[dashed] ([c]CM)--([c]Ct);
       \draw[dashed] ([c]Ew)--([c]Ewt);
       \draw[dashed] ([c]AM)--([c]Dt);
     \draw[dashed] ([c]BM)--([c]BMt);

        \foreach \x in {A,B,AM,BM,At,Bt,Ct,Dt,BMt,CM,Ewt}
   \fill ([c]\x) circle (3pt);
     	\fill[red] ([c]Ew) circle (3pt);

    \end{scope}

%%%%%%%%%%%%%%%%%%%%%%%%%%%%%%%%%%%%%%%%%%%%%%%%%%%%%%

%%%%%%%%%%%%%%%%%%%%%%%%%%%%%%%%%%%%%%%%%%%%%%%%%%

%Theta graph case
\coordinate (Cc) at (0,0.5,0); %need for Weierstrass
\coordinate (C1) at (0,1,0);
\coordinate (C1b) at (0,1,-.8);
\coordinate (A1) at (1.5,3,0);
\coordinate (A1b) at (1.5,3,-.8);

%3rd theta type
\coordinate (G) at (0.2,2.75,0);
\coordinate (F1) at (2,3,0);
\coordinate(F2) at (2,2.5,0);
% 4th Theta Graph 
\coordinate (A14) at (1.2,4,0);
\coordinate (H1) at (3.5,2.2,3.2);
\coordinate (H2) at (3.5,2.2,4);

\coordinate (B1) at (3.5,3.7,-.8);
\coordinate (B2) at (3.5,3.7,0);
\coordinate (A2) at (8,5,7);
\coordinate (A3) at (5.8,3.9,0);

%target

\coordinate (C1t) at (0,0,0);
%target %3rd theta type
\coordinate (Gt) at (0.2,0,0);
\coordinate (Ft) at (2,0,0);

%target 4th theta type

\coordinate (A14t) at (1.2,0,0);
\coordinate (H1t) at (3.5,0,4);

\coordinate (A1t) at (1.5,0,0);
\coordinate (B1t) at (3.5,0,0);
\coordinate (A2t) at (8,0,7);
\coordinate (A3t) at (5.8,0,0); 

%branch points
\coordinate (Q1) at ( 1.2,0,1.5);
\coordinate (Q2) at ( 1.8,0,1.8);
\coordinate (Q3) at ( 8.3,0,8.9);
\coordinate (Q4) at (7.6,0,8);
\coordinate (Q5) at ( 5.2,0,1.5);
\coordinate (Q61) at ( 5.8,0,1.8);

\begin{scope}[every coordinate/.style={shift={(10,10,0)}}]

%%%%% theta 1

\draw[->-] ([c]C1b)--([c]A1b);
\draw[->-] ([c]C1)--([c]A1);
\draw ([c]A1) -- ([c]A1b);
\draw[->-] ([c]A1b) to [out=70,in=150] ([c]B1);
\draw[->-] ([c]A1) to [out=-70,in=-70] ([c]B2);
%\draw ([c]B1)--([c]A3M1);
%\draw ([c]B2)--([c]A3);
\draw ([c]B1) to [out=10,in=150] ([c]A3);
\draw ([c]B2) to [out=-10,in=-150] ([c]A3);
%\draw ([c]B1)--([c]A2b);
\draw ([c]B1) to [out=-40,in=110] ([c]A2);
\draw ([c]B2) to [out=-50,in=-100] ([c]A2);
%%%%%%%%%%%%%%%%%%%%%
     		%target
\draw[->-] ([c]C1t)--node[above] {\small{$2$}}([c]A1t);
\draw[->-] ([c]A1t)--node[above] {\small{$1$}}([c]B1t);
\draw ([c]B1t)--node[right] {\small{$0$}}([c]A2t);
\draw ([c]B1t)--node[above] {\small{$0$}}([c]A3t);

%branched points

		\draw[blue] ([c]A1t)--([c]Q1);
		\draw[blue] ([c]A1t)--([c]Q2);
		\draw[blue] ([c]A2t)--([c]Q3);
		\draw[blue] ([c]A2t)--([c]Q4);
		\draw[blue] ([c]A3t)--([c]Q5);
		\draw[blue] ([c]A3t)--([c]Q61);

 %admissible cover
       
       \draw[dashed] ([c]C1)--([c]C1t);
        \draw[dashed] ([c]C1b)--([c]C1t);
       \draw[dashed] ([c]A1)--([c]A1t);
       \draw[dashed] ([c]A1b)--([c]A1t);
       \draw[dashed] ([c]B1)--([c]B1t);
        \draw[dashed] ([c]B2)--([c]B1t);
        \draw[dashed] ([c]A2)--([c]A2t);
        \draw[dashed] ([c]A3)--([c]A3t);

\foreach \x in {A1,A1b,A2,A3,B1,B2,C1t,A1t,B1t,A2t,A3t}
   \fill ([c]\x) circle (3pt);

  \fill[red] ([c]C1) circle (3pt);
    \fill[red] ([c]C1b) circle (3pt);

\end{scope}

\begin{scope}[every coordinate/.style={shift={(10,1,0)}}]

%%%%% theta 2

%\draw[->-] ([c]C1)--([c]A1);
\draw[->-] ([c]Cc)--([c]A1);
\draw[->-] ([c]A1) to [out=70,in=150] ([c]B1);
\draw[->-] ([c]A1) to [out=-70,in=-70] ([c]B2);

\draw ([c]B1) to [out=10,in=150] ([c]A3);
\draw ([c]B2) to [out=-10,in=-150] ([c]A3);
%\draw ([c]B1)--([c]A2b);
\draw ([c]B1) to [out=-40,in=110] ([c]A2);
\draw ([c]B2) to [out=-50,in=-100] ([c]A2);
%\draw ([c]B1)--([c]A3M1);
%\draw ([c]B2)--([c]A3);
%\draw ([c]B1) to [out=10,in=75] ([c]A3);
%\draw ([c]B2) to [out=-10,in=-75] ([c]A3);
%\draw ([c]B1)--([c]A2b);
%\draw ([c]A2b) to [out=55,in=120] ([c]A2);
%\draw ([c]B2) to [out=-100,in=-80] ([c]A2);
%%%%%%%%%%%%%%%%%%%%%
     		%target
\draw[->-] ([c]C1t)--node[above] {\small{$\frac{3}{2}$}}([c]A1t);
\draw[->-] ([c]A1t)--node[above] {\small{$1$}}([c]B1t);
\draw ([c]B1t)--node[right] {\small{$0$}}([c]A2t);
\draw ([c]B1t)--node[above] {\small{$0$}}([c]A3t);

%branched points

		\draw[blue] ([c]A1t)--([c]Q1);
		\draw[blue] ([c]A1t)--([c]Q2);
		\draw[blue] ([c]A2t)--([c]Q3);
		\draw[blue] ([c]A2t)--([c]Q4);
		\draw[blue] ([c]A3t)--([c]Q5);
		\draw[blue] ([c]A3t)--([c]Q61);

\foreach \x in {A1,A2,A3,B1,B2,C1t,A1t,B1t,A2t,A3t}
   \fill ([c]\x) circle (3pt);
  
     		%\fill[red] ([c]C2) circle (3pt);

 %admissible cover
       
       \draw[dashed] ([c]Cc)--([c]C1t);
        %\draw[dashed] ([c]C2)--([c]C1t);
       \draw[dashed] ([c]A1)--([c]A1t);
       \draw[dashed] ([c]B1)--([c]B1t);
        \draw[dashed] ([c]B2)--([c]B1t);
        \draw[dashed] ([c]A2)--([c]A2t);
        \draw[dashed] ([c]A3)--([c]A3t);
        
         \fill[red] ([c]Cc) circle (3pt);
\end{scope}

\begin{scope}[every coordinate/.style={shift={(10,-8,0)}}]

%%%% theta 3 (last)

\draw ([c]B1) to [out=170,in=30] ([c]A14);
\draw ([c]B2) to [out=190,in=-30] ([c]A14);

\draw[->-] ([c]H1) to ([c]B1);
\draw[->-] ([c]H2) to ([c]B2);
\draw ([c]B1) to [out=10,in=150] ([c]A3);
\draw ([c]B2) to [out=-10,in=-150] ([c]A3);

\draw ([c]B1) to [out=-40,in=110] ([c]A2);
\draw ([c]B2) to [out=-50,in=-100] ([c]A2);

%%%%%%%%%%%%%%%%%%%%%
     		%target
\draw ([c]A14t)--node[above] {\small{$0$}}([c]B1t);
\draw[->-] ([c]H1t)--node[right] {\small{$2$}}([c]B1t);
\draw ([c]B1t)--node[right] {\small{$0$}}([c]A2t);
\draw ([c]B1t)--node[above] {\small{$0$}}([c]A3t);

%branched points

		\draw[blue] ([c]A14t)--([c]Q1);
		\draw[blue] ([c]A14t)--([c]Q2);
		\draw[blue] ([c]A2t)--([c]Q3);
		\draw[blue] ([c]A2t)--([c]Q4);
		\draw[blue] ([c]A3t)--([c]Q5);
		\draw[blue] ([c]A3t)--([c]Q61);

\foreach \x in {A14,A2,A3,B1,B2,H1t, A14t,B1t,A2t,A3t}
   \fill ([c]\x) circle (3pt);
  
     		%\fill[red] ([c]C2) circle (3pt);

 %admissible cover
       
       \draw[dashed] ([c]A14)--([c]A14t);
        %\draw[dashed] ([c]C2)--([c]C1t);
       \draw[dashed] ([c]H1)--([c]H1t);
       \draw[dashed] ([c]H2)--([c]H1t);
     \draw[dashed] ([c]B1)--([c]B1t);
        \draw[dashed] ([c]B2)--([c]B1t);
        \draw[dashed] ([c]A2)--([c]A2t);
        \draw[dashed] ([c]A3)--([c]A3t);
        
 \fill[red] ([c]H1) circle (3pt);
   \fill[red] ([c]H2) circle (3pt);
\end{scope}
\end{tikzpicture}
	%\end{minipage}
\caption{Tropicalisation of semistable tails with maximally degenerated core: the dumbbell (l), and the theta graph (r). The red vertices correspond to the special branches.}
\label{fig:admredcores}
\end{figure}
%\clearpage
 
 Summing up, we have proven the following:

\begin{prop}\label{prop:tailI}
 Let $\phi\colon\mathcal C\to\overline{\mathcal C}$ be a birational contraction over the spectrum of a discrete valuation ring $\dvr$, where: $\mathcal C\to \dvr$ is a family of prestable (reduced, nodal) curves of arithmetic genus two with smooth generic fibre $\mathcal C_{\eta}$; $\overline{\mathcal C}\to\dvr$ is a family of Gorenstein curves with a genus two singularity of type $I_m$ (resp. $I\!I_m$) at $q\in\overline{\mathcal C}_0$. Denote by $(Z;q_1,\ldots,q_m)$ the exceptional locus $\Exc(\phi)=\phi^{-1}(q)$, marked with $Z\cap\overline{\mathcal C_0\setminus Z}$, where $q_m$ corresponds to the singular branch of $\overline{\mathcal C}_0$ (resp. $q_1,q_m$ correspond to the twin branches of $\overline{\mathcal C}_0$). Then:
 \begin{enumerate}[leftmargin=.6cm]
  \item $(Z,q_m)$ is Weierstrass (resp. $(Z,q_1,q_m)$ is conjugate).
  \item On $\operatorname{trop}(\mathcal C)$, the distance of $q_m$ (resp. $q_1$ and $q_m$ - they are equidistant) from the core is less than the distance of any other $q_i$ from the core, and the former - together with the shape of $\operatorname{trop}(Z)$ - determines the latter.
 \end{enumerate}
\end{prop}

Vice versa, every such genus two subcurve can be contracted to a Gorenstein singularity.

\begin{prop}\label{prop:contractionI}
 Let $(\mathcal C,p_1,\ldots,p_n)\to\dvr$ be a family of pointed semistable curves of arithmetic genus two such that $\mathcal C$ has regular total space and smooth generic fibre, and $(\mathcal C,p_1)\to \Delta$ is Weierstrass (resp. $(\mathcal C,p_1,\bar p_1)\to \Delta$ is conjugate). Let $(Z,q_1,\ldots,q_m)$ be a genus two subcurve of $\mathcal C_0$ containing none of the $p_i(0)$, marked by $Z\cap \overline{\mathcal C_0\setminus Z}$ so that the tail containing $p_1$ is attached to $Z$ at $q_m$ (resp. the tails containing $p_1$ and $\bar p_1$ are attached to $Z$ at $q_1$ and $q_m$), and satisfying all the shape prescriptions above. There exists a contraction $\phi\colon\mathcal C\to\overline{\mathcal C}$ over $\dvr$, with exceptional locus $Z$, such that $\overline{\mathcal C}\to\dvr$ is a family of Gorenstein curves containing a type $I_m$ (resp. type $I\!I_m$) singularity in the central fibre.
\end{prop}

\begin{proof}(of Proposition \ref{prop:contractionI})
 By blowing down some rational tails outside $Z$, we can assume that $\mathcal C_0\setminus Z=\sqcup_{i=1}^m T_i$ with each $T_i\simeq\PP^1$. The image of $p_i(0)$ and $p_j(0)$ might now coincide for $i\neq j$. The total space of the curve can still be assumed to be smooth by the Castelnuovo criterion. By abuse of notation, we denote the resulting family of pointed curves by $(\mathcal C,p_1,\ldots,p_n)$. By assumption on the shape of $Z$, we can find an effective Cartier $\pazocal Z$ supported on $Z$ such that $\mathcal L=\omega_{\mathcal C/\dvr}(\pazocal Z+\sum p_i)$ is trivial on $Z$ and relatively ample elsewhere (both on $T_i$ and on the generic fibre). Now we show that $\mathcal L$ is semiample on $\mathcal C$.
 
 Consider the (a priori different) line bundle $\mathcal L^\prime=\OO(2p_1+\sum p_i)$ (resp. $\OO(p_1+\bar p_1+\sum p_i)$). Since we assumed $p_1$ to be Weierstrass (resp. $p_1$ and $\bar p_1$ to be conjugate), $\mathcal L_\eta\simeq\mathcal L^\prime_\eta$. On the other hand it is easy to see that the multi-degrees of $\mathcal L_0$ and $\mathcal L^\prime_0$ coincide, as $Z$ is unmarked and each rational tail is isomorphic to $\PP^1$; it follows from the separatedness of $\Pic^0_{\mathcal C/\dvr}\to\dvr$ (see \cite[p. 136]{Deligne-Gabber} or \cite[\S 9.4]{BLR}) that $\mathcal L$ and $\mathcal L^\prime$ are isomorphic line bundles, so that, in particular, $\mathcal L$ is trivial \emph{on a neighbourhood of $Z$}. Observe now that 
 \[R^1\pi_*\mathcal L(-{\pazocal Z})= R^1\pi_*\omega_{\mathcal C/\dvr}(\sum p_i)=0\]
 by semistability, hence $\pi_*\mathcal L\twoheadrightarrow \pi_*(\mathcal L_{|{\pazocal Z}})=\pi_*\OO_{\pazocal Z}$, which contains the constants, showing that $\mathcal L$ is semiample along $Z$; that it is along the $T_i$ is easier.
 
 We therefore have a well-defined  morphism:
 \[\mathcal C\xrightarrow{\phi}\overline{\mathcal C}=\underline{\operatorname{Proj}}_\dvr\left(\bigoplus_{n\geq 0}\pi_*\mathcal L^{\otimes n}\right)\to\dvr\]
 associated to $\mathcal L$. The proof that $\overline{\mathcal C}\to \dvr$ is a flat family of Gorenstein curves goes along the lines of \cite[Lemma 2.13]{SMY1} or \cite[Proposition 3.7.3.1]{RSPW1}. It is then clear from the classification that it contains a type $I_m$ (resp. $I\!I_m$) singularity.
 
 %The proof of Proposition \ref{prop:contractionII} is entirely analogous.
\end{proof}

\begin{rem}\label{rem:smoothable_sing}
 It follows that genus two Gorenstein singularities are smoothable.
\end{rem}

It would be interesting to construct the contraction of Proposition \ref{prop:contractionI} pointwise - as opposed to in a smoothing family - by extending the methods of \cite{Bozlee}.

\section{The new moduli functors}\label{sec:stability}
The idea is to replace subcurves of positive genus with isolated singularities, the number of special points on the former bounding the number of branches of the latter. The following is a slight generalisation of \cite[Definition 3.4]{SMY1}.
\begin{definition}
 Let $(C,p_1,\ldots,p_n)$ be a reduced curve, marked by smooth points. For a connected subcurve $D\subseteq C$, we define its \emph{level} to be: \[ \lev(D)=\lvert D\cap\overline{C\setminus D}\rvert+\lvert\{p_1,\ldots,p_n\}\cap D\rvert.\]
\end{definition}
In this definition, the multiplicity of $D\cap\overline{C\setminus D}$ is not taken into account.

We omit the proof of the following lemma; compare with \cite[Corollary 3.2, Lemma 3.5]{SMY1}.
\begin{lem}
 Let $(C,p_1,\ldots,p_n)$ be a pointed \emph{semistable} curve of arithmetic genus two, with minimal genus two subcurve $Z$. For every connected subcurve $Z^\prime\subseteq C$ of genus two, we have an inclusion $Z\subseteq Z^\prime$ and $\lev(Z)\leq\lev(Z^\prime)$.
\end{lem}

\begin{definition}
We say that a point $p$ \emph{cleaves} to a component $D$ of a curve $C$ if there is a \emph{unique} semistable rational chain of length $k\geq 0$ (see the discussion preceding Lemma \ref{lem:min2}) in $C$ connecting $p$ to a smooth point of $D$. 
\end{definition}

\begin{rem}
 Allowing singularities of genus two forces us to allow singularities of genus one as well by deformation openness. Indeed, singularities of genus zero and one appear in the miniversal family of singularities of genus two. Also, singularities of type $I$ do appear in the miniversal family of singularities of type $I\!I$, and vice versa. For low values of $m$, this follows from a neat result of Grothendieck (\cite[p. 2277]{C-ML}, see also \cite{Arnold,Demazure}):
 \begin{thm}\label{thm:ADE}
  Let $(C,q)$ be a curve singularity of ADE type. The singularities appearing in the miniversal deformation of $(C,q)$ are all and only the ADE singularities whose Dynkin diagram can be obtained as a full subgraph of the Dynkin diagram of $(C,q)$.
 \end{thm}
\end{rem}

We finally come to the definition of $m$-stability for curves of genus two.
\begin{definition}\label{def:m-stability}
 Fix positive integers $1\leq m<n$. Let $(C,p_1,\ldots,p_n)$ be a connected, reduced, complete curve of arithmetic genus two, marked by smooth distinct points. We say that $C$ is $m$-stable if:
 \begin{enumerate}[leftmargin=.7cm]
  \item\label{cond:sing} $C$ is Gorenstein with only: nodes; elliptic $l$-fold points, $l\leq m+1$; type $I_{\leq m}$, type $I\!I_{\leq m}$, and dangling (see Definition \ref{def:dangling}) $I\!I_{m+1}$ singularities of genus two, as singular points.
  \item\label{cond:lev2} If $Z$ is a connected subcurve of arithmetic genus two, then $\lev(Z)>m$.
  \item\label{cond:lev1} If $E$ is a connected subcurve of arithmetic genus one, then either $\lev(E)>m+1$,
  or $p_1$ cleaves to $E$ and $\overline{C\setminus E}$ is a union of rational curves.
 %or $E$ is not nodally attached and $p_1$ cleaves to it.
  \item\label{cond:aut} $H^0(C,\Omega_C^\vee(-\sum_{i=1}^n p_i))=0$.
  \item\label{cond:p1} If $C$ contains a singularity of genus two, or an elliptic $l$-fold point with a self-branch or a genus one branch, then $p_1$ cleaves to one of the special branches (see Remark \ref{rem:special}).
 \end{enumerate}
\end{definition}

\begin{rem}
 The definition is not $\mathfrak{S}_n$-symmetric. In the argument below, we exploit the asymmetry to write the dualising line bundle of a genus two (sub)curve $Z$ as $\omega_Z\simeq\OO_Z(q_1+\bar q_1)$, where $q_1$ is the point of $Z$ closest to $p_1$, and $\bar q_1$ its conjugate, sometimes depending on a one-parameter smoothing. Compare with the situation in genus one, where the dualising line bundle of a minimal Gorenstein curve is trivial (all smooth points are non-special). We also refer to $p_1$ when deciding which genus one subcurve to contract first in case there are two disjoint ones of low level.
\end{rem}

\begin{rem}
 The case $m=0$ would not give back the Deligne-Mumford compactification, but rather Schubert's one.
\end{rem}

\begin{rem}\label{rmk:lev1solev2}
 If there is a nodally attached subcurve of genus one, condition \eqref{cond:lev1} and condition \eqref{cond:aut} jointly imply condition \eqref{cond:lev2}. Indeed, from Corollary \ref{cor:explicitnoaut} we have $\lev(Z)\geq\lev(E)-1$. The only cases (up to relabelling) in which the level drops by one are: when $Z=(E,p_1,\ldots,p_{l-2},q_1,q_2)\sqcup_{\{q_1,q_2\}}(\PP^1,q_1,q_2,p_{l-1})$; and when $Z=(E_1,p_1,\ldots,p_{l-1},q)\sqcup_q(E_2,q)$, where all the $E$ have genus one.
\end{rem}

The following is our main:
{\setstretch{1.1}
\begin{thm}
 For $1\leq m <n$, the moduli stack of $n$-pointed $m$-stable curves of genus two $\oM_{2,n}^{(m)}$ is a proper Deligne-Mumford stack of finite type over $\operatorname{Spec}(\mathbb Z[\frac{1}{6}])$ - containing $\pazocal M_{2,n}$ as a dense open substack, and therefore irreducible.
\end{thm}}
\begin{proof}
 \begin{enumerate}
  \item \emph{Algebraicity (diagonal)} - \emph{The diagonal $\Delta\colon \oM_{2,n}^{(m)}\to \oM_{2,n}^{(m)}\times\oM_{2,n}^{(m)}$ is representable, quasicompact, and of finite type.} Since $m$-stable curves are canonically polarised, it follows from Grothendieck's theory of Hilbert schemes that the \emph{Iso}-functor between two $m$-stable curves over $S$ is representable by a quasiprojective scheme over $S$.
  
  \item \emph{Algebraicity (atlas) \& irreducibility} - \emph{There exists an irreducible scheme $H$, of finite type over $\operatorname{Spec}(\mathbb Z[\frac{1}{6}])$, with a smooth and surjective morphism $H\to \oM_{2,n}^{(m)}$.} Fix an integer $N>2+8(m+1)$; let $d=N(2+n)$ and $r=d-2$. By Lemma \ref{lem:boundedness} below, every $n$-pointed $m$-stable curve over a field $\k$ admits a pluri-log-canonical embedding of degree $d$ in $\PP^r$. Let $H_0$ denote the Hilbert scheme of degree $d$, genus two curves in $\PP^r$. Let $H_1\subseteq H_0\times(\PP^r)^{\times n}$ denote the locally closed subscheme consisting of $([C],p_1,\ldots,p_n)$ such that every $p_i$ belongs to the smooth locus of $C$; this is open in the incidence variety. By Lemma \ref{lem:defop} below, there is an open subscheme $H_2\subseteq H_1$ parametrising $m$-stable curves (note that $H_1$ is of finite type over $\operatorname{Spec}(\mathbb Z[\frac{1}{6}])$, and in particular Noetherian). By the representability of the Picard scheme \cite[Proposition 5.1]{GIT}, there is a locally closed subscheme $H\subseteq H_2$ representing $([C],p_1,\ldots,p_n)$ such that $\OO_C(1)=\omega_C(\sum_{i=1}^np_i)^{\otimes N}$. Now there is a morphism $H\to \oM_{2,n}^{(m)}$ that is surjective by construction, and smooth because two different embeddings of an $m$-stable curve differ by the choice of a basis of $H^0(C,\omega_C(\sum_{i=1}^np_i)^{\otimes N})$. Since every $m$-stable curve is smoothable (Remark \ref{rem:smoothable_sing} and \cite[I.6.10]{Kollar-rational}), $H$ is irreducible.
  
  \item \emph{DM} - \emph{The diagonal $\Delta\colon \oM_{2,n}^{(m)}\to \oM_{2,n}^{(m)}\times\oM_{2,n}^{(m)}$ is unramified.} It is enough to show that, for an $m$-stable curve $(C,p_1,\ldots,p_n)$ over a field $\k$, the \emph{Iso}-group scheme $\operatorname{Aut}_{\k}(C,p_1,\ldots,p_n)$ is unramified. Its tangent space at the identity can be identified with the vector space $H^0(C,\Omega_C^\vee(-\sum_{i=1}^n p_i))$, which vanishes by Definition \ref{def:m-stability}\ref{cond:aut}. Note that we need the assumption on the base characteristic in order to translate this vanishing into a combinatorial criterion on the pointed normalisation and singularity type of the curves (Corollary \ref{cor:explicitnoaut}).
  
  \item \emph{Properness} - Follows from the valuative criterion (Proposition \ref{prop:valcrit}).
 \end{enumerate}

\end{proof}

\begin{lem}[(boundedness)]\label{lem:boundedness}
 If $(C,p_1,\ldots,p_n)$ is an $m$-stable curve of genus two, the $N$-th power of $A=\omega_C(\sum_{i=1}^np_i)$ is very ample for every $N>2+8(m+1)$.
\end{lem}
\begin{proof}
 It is enough to show that, for every pair of points $p,q\in C$ (possibly equal):
 \begin{enumerate}
  \item \emph{basepoint-freeness}: $H^1(C,A^{\otimes N}\otimes I_p)=0$;
  \item \emph{separating points and tangent vectors}: $H^1(C,A^{\otimes N}\otimes I_pI_q)=0$.
 \end{enumerate}
By Serre duality we may equivalently show that $H^0(C,\omega_C\otimes A^{-N}\otimes(I_pI_q)^\vee)=0$. Let $\nu\colon\tilde C\to C$ be the normalisation, and let $\nu^{-1}(p)=\{p_1,\ldots,p_h\}$, $\nu^{-1}(q)=\{q_1,\ldots,q_k\}$, with $h,k\leq m+1$. It follows from Proposition \ref{prop:classification} and \cite[Proposition A.3]{SMY1} that $\nu_*\OO_{\tilde C}(-D)\subseteq I_pI_q$ for $D=4(\sum_{i=1}^hp_i+\sum_{j=1}^kq_j)$ (note that $\deg(D)\leq 8(m+1)$); furthermore, the quotient is torsion, therefore, by applying $\hhom(-,\OO_C)$, we find $(I_pI_q)^\vee\subseteq\nu_*\OO_{\tilde C}(D)$. It is thus enough to show that $H^0(\tilde C,\OO_{\tilde C}(D)\otimes\nu^*(\omega_C\otimes A^{-N}))=0$. Finally, $\nu^*\omega_C$ (resp. $\nu^*A$) has degree at most two (resp. at least one) on each component of $\tilde C$, hence it is enough to take $N>2+8(m+1)$.
\end{proof}

\begin{lem}[(deformation openness)]\label{lem:defop}
 Let $(\mathcal C,\sigma_1,\ldots,\sigma_n)\to S$ be a family of curves over a Noetherian base scheme with $n$ sections. The locus \[\{s\in S|(\mathcal C_{\bar s},\sigma_1(\bar s),\ldots,\sigma_n(\bar s)) \text{ is } m\text{-stable}\}\] is Zariski-open in $S$.
\end{lem}
\begin{proof}
 Having connected fibres which are Gorenstein curves of arithmetic genus two is an open condition (see for example \cite[\href{https://stacks.math.columbia.edu/tag/0E1M}{Tag 0E1M}]{stacks-project}). Only singularities of genus zero (nodes), one (elliptic $l$-folds), and two may then occur.
 
 The case $m=1$ deserves special attention. In this case, that condition \eqref{cond:sing} is open follows from acknowledging that $I_1=A_4$, $I\!I_2=A_5$, while tacnodes, cusps, and nodes are $A_3$, $A_2$, and $A_1$-singularities respectively, and from Grothendieck's result on the deformation theory of ADE singularities (see Theorem \ref{thm:ADE} above).
 
 The case $m\geq 2$ simply follows from upper semicontinuity of embedded dimension and the fact that we have exhausted all possible Gorenstein singularities of genus $\leq 2$, and embedding dimension $\leq m+1$.
 
 Condition \eqref{cond:aut} translates to: the locus where the automorphism group is unramified is open in the base. But this holds more generally for group schemes (see the end of the proof of \cite[Lemma 3.10]{SMY1}): suppose that $p\colon G\to S$ is unramified at $g\in G$; then, it is unramified in a neighbourhood $g\in U\subseteq G$. Translating $U$ we can make sure that it is saturated with respect to $p$, so that we can transfer the openness of the unramified locus from the source to the target of $p$.
 
 The other conditions are topological, hence constructible. Since $S$ is Noetherian, it is enough to check their openness over the spectrum of a discrete valuation ring. Assume that the geometric generic fibre $C_{\bar\eta}$ contains two genus one subcurves $E_{1,\bar\eta}$ and $E_{2,\bar\eta}$; their closures $E_1$ and $E_2$ in $\mathcal C$ are then flat families of genus one curves over $\dvr$. %If $E_{1,\bar\eta}$ and $E_{2,\bar\eta}$ have empty intersection, then so do $E_1$ and $E_2$, by the local constancy of the number of connected components (from the Zariski decomposition and \cite[\href{https://stacks.math.columbia.edu/tag/0E0D}{Tag 0E0D}]{stacks-project}). If $E_{1,\bar\eta}$ and $E_{2,\bar\eta}$ are joined by a disconnecting node $q_{\bar\eta}$, then so are $E_{1,0}$ and $E_{2,0}$; indeed, the unique limit of $q_{\bar\eta}$ must be a singular point of the projection, but cannot be any worse than a node (because we have already used up all of our genus allowance). Finally, if $E_{1,\bar\eta}$ and $E_{2,\bar\eta}$ share a branch, then so do $E_{1,0}$ and $E_{2,0}$; on the other hand, if $E_{i,\bar\eta}$ has more than one branch, then so does $E_i$. Similarly, if $C_{\bar\eta}$ contains only one subcurve of genus one, with two nodes joined by a rational chain, so does $C_0$. To summarise,
 The number of connected components of $\overline {C\setminus E_i}$ is locally constant (by the Zariski decomposition and \cite[\href{https://stacks.math.columbia.edu/tag/0E0D}{Tag 0E0D}]{stacks-project}), so: \[\lvert E_{i,\bar\eta}\cap\overline{C_{\bar\eta}\setminus E_{i,\bar\eta}}\rvert=\lvert E_{i,0}\cap\overline{C_{0}\setminus E_{i,0}}\rvert.\]
 The number of markings on $E_i$ is also constant. Hence we can deduce condition \eqref{cond:lev1} for $C_{\bar\eta}$ from the same condition on $C_0$. Condition \eqref{cond:lev2} follows in this case from Remark \ref{rmk:lev1solev2}; it can be proven analogously when there is no subcurve of genus one.
 
 Finally, suppose that $C_{\bar\eta}$ has a genus two singularity, then so does $C_0$. The (union of the) distinguished branch(es) $E_{\bar\eta}$ of $C_{\bar\eta}$ is a genus one singularity, and so is its limit $E_0$ in $C_0$. It has to contain the distinguished branch(es) of $C_0$, because any subcurve not containing them has genus zero; therefore, by assumption, $E_0$ contains $p_{1,0}$. Then also $E_{\bar\eta}$ contains $p_{1,\bar\eta}$, because the markings are contained in the non-singular locus of the curve. Similarly, if $C_{\bar\eta}$ has a genus one singularity with a self-branch, the limit of such a branch is a genus one subcurve $E_0$ of $C_0$; the latter may very well acquire a genus two singularity, but $E_0$ will contain the special branches of it, so it will be connected to $p_1$. We conclude as above. The case that $C_{\bar\eta}$ contains a genus one subcurve of low level is analogous. We have thus proved that condition \eqref{cond:p1} is open.
\end{proof}

\begin{prop}[(Valuative criterion of properness for $\oM^{(m)}_{2,n}$)]\label{prop:valcrit}
 Given a smooth $n$-pointed curve of genus two $C_\eta$ over a discrete valuation field $\eta=\operatorname{Spec}(K)\hookrightarrow\dvr$, there exists a finite base-change $\dvr^\prime\to\dvr$ after which $C_\eta$ can be completed to an \emph{$m$-stable} curve over $\dvr^\prime$. Two such models are always dominated by a third one.
\end{prop}

\begin{proof}[Existence of limits]

By properness of the moduli space of pointed admissible covers, after a finite base-change $\dvr^\prime\to\dvr$ we can complete $C_\eta$ to a prestable curve $\mathcal C^\prime\to\dvr^\prime$ together with an admissible cover $\mathcal C^\prime\to\mathcal T^\prime$ over $\dvr^\prime$. We drop the primes from the notation. Let $\psi\colon \tropC\to \tropT$ be the tropicalisation of the admissible cover, as in Section \ref{sec:recaplog}.

If there are two disjoint subcurves of arithmetic genus one $E_1$ and $E_2$ in $\mathcal C_0$, either they already satisfy condition \eqref{cond:lev1} of $m$-stability, or we proceed as in \cite{SMY1,RSPW1}: we draw circles around them and let the radius increase. If we can make the outer valence $l_+$ of the circle be at least $m+2$, while the inner valence $l_-$ is at most $m+1$, then we can contract the strict interior of the disc by \cite[Lemma 2.13]{SMY1} or \cite{Bozlee}, and we will get an elliptic $l_-$-fold point of level $l_+$. In general, when the circle passes through a rational vertex, Deligne-Mumford (semi)stability of $\mathcal C_0$ ensures that the inner valence stays the same, while the outer valence can only increase. Note that if $p_1$ cleaves to only one of the two elliptic subcurves, say $E_1$, then we should start by inflating the disc around $E_2$, since if this gets to touch $E_1$, the latter is not required to satisfy any level condition in the contraction (by the second clause of \eqref{cond:lev1}). Similarly, if $p_1$ cleaves to a rational component $R$ on the bridge between $E_1$ and $E_2$, if the circles meet at the vertex corresponding to $R$, contracting the strict interior produces two elliptic singularities with a common branch $R$. No level condition is then required of them individually, but the genus two core must still satisfy condition \eqref{cond:lev2} of $m$-stability. If it does not, we contract it to a genus two singularity as follows.

\begin{rem}
 If the two minimal elliptic subcurves were circles of $\PP^1$'s sharing a branch $R$ to which $p_1$ clove, the level condition \eqref{cond:lev1} would not apply, and we would proceed directly as follows.
\end{rem}

We are assuming now that the genus two core does not satisfy the level condition \eqref{cond:lev2}. By condition \eqref{cond:p1} of $m$-stability, if there is a genus two singularity in the contraction, then $p_1$ must cleave to the special branch. This determines the shape of the exceptional locus according to Proposition \ref{prop:tailI}. Indeed, the position of $p_1$ determines a piecewise-linear function $\lambda\colon\tropC\to\mathbb R$ (that we think of as a height function, compare with the level graphs of \cite{BCGGM}).

We actually construct a piecewise-linear function $\lambda_T$ on $\tropT$; $\lambda$ is its pullback along $\psi$. We call the \emph{core} of $\tropT$ the image of the core of $\tropC$. Up to a global translation by $\mathbb R$, the function $\lambda_T$ is characterised by having:
\begin{itemize}
 \item slope $2$ or $\frac{3}{2}$ towards the core on the edges separating $p_1$ from the core, according to whether they are ``conjugate'' or ``Weierstrass'' (i.e. whether the admissible cover is ramified or not over them), and slope $2$ along the infinite leg corresponding to $p_1$ (by stability of pointed admissible covers, we are assuming that $p_1$ itself is not a Weierstrass point);
 \item slope $1$ or $\frac{1}{2}$ towards the core on every other edge and infinite leg outside the core, according to whether they are ``conjugate'' or ``Weierstrass'';
 \item slope on the core determined by balancing \eqref{eqn:Dv} as in Figure \ref{fig:admredcores}.
\end{itemize}
(We may fix the value of $\lambda_T$ by saying its maximum is $0$, although this choice is both arbitrary and irrelevant.) %Let $\lambda$ be the pullback piecewise-linear function on $\tropC$ (slopes get multiplied by the expansion factor).

 The subcurve to be contracted is of the form $\lambda^{-1}(\mathbb R_{>\rho})$, where $\rho$ is the value attained by $\lambda$ on a vertex of $\tropC$, such that there are $\leq m$ edges leaving $\lambda^{-1}(\{\rho\})$ in the upward direction, and $\geq m+1$ leaving it downwards. Such a vertex can be found because there are $n\geq m+1$ vertices at height $-\infty$ (corresponding to the infinite legs), and less than $m+1$ just below the core (since we assumed that the core does not satisfy the level condition \eqref{cond:lev2}). Cutting $\lambda$ off at level $\rho$, i.e. setting $\mu:=\max\{\lambda-\rho,0\}$, and subdividing $\tropC$ according to the domain of linearity of $\mu$, provides a partial destabilisation $\widetilde{\mathcal C}\to\mathcal C$, and a honestly piecewise-linear function $\mu$ on $\ttropC$.
 The curve:
 
 \[\oC=\underline{\operatorname{Proj}}_\dvr\left(\tilde{\pi}_*\bigoplus_{d\geq 0}\omega_{\tC/\dvr}(p_1+\ldots+p_n)(\mu)^{\otimes d}\right)\]
 contains a genus two singularity with less than $m$ branches and more than $m+1$ special points in the central fibre; it is endowed with a birational contraction $\phi\colon\tC\to\oC$, see Proposition \ref{prop:contractionI}. Upon contracting any rational tail away from the singularity, $\oC_0$ is the $m$-stable limit of $C_\eta$.

\end{proof}

\begin{proof}[Uniqueness of limits]

 Suppose that $\cC\to\dvr$ and $\cC^\prime\to\dvr^\prime$ are $m$-stable limits of $C_\eta$. Up to a further base-change (and a slight abuse of notation), there is a diagram:
 \bcd
 & \mathcal C^{ss}\ar[ld,"\phi" above]\ar[dr,"\phi^\prime" above] & \\
 \mathcal C\ar[dr] & & \mathcal C^\prime\ar[dl] \\
 & \dvr &
 \ecd
 extending the isomorphism between the generic fibres, where $\mathcal C^{ss}$ has semistable central fibre and regular total space, by the semistable reduction theorem. We may also assume that there is a hyperelliptic admissible cover $\mathcal C^{ss}\to\mathcal T$, and that there is a piecewise-linear function $\lambda^\prime$ on the tropicalisation $\tropC$ of $\mathcal C^{ss}$ such that $(\phi^\prime)^*\omega_{\mathcal C^\prime}=\omega_{\mathcal C}(\lambda^\prime)$, see Proposition \ref{prop:tailI}. Our goal is to show that the exceptional loci of $\phi$ and $\phi^\prime$ are the same, and conclude by \cite[Lemma 1.15]{Debarre}.
 
 Suppose that $\mathcal C_0$ contains an elliptic $l$-fold point $x$. Set $E_x=\phi^{-1}(x)$, which is a balanced connected subcurve of $\mathcal C^{ss}_0$, with arithmetic genus one and level $l\leq m+1$. If we are in the situation of Lemma \ref{lem:min2}, \ref{case:twice1} or \ref{case:2}, so $x$ admits a special branch $X$, then $p_1$ has to cleave to $X$ by \eqref{cond:p1} of $m$-stability, so in particular it does not cleave to $E_x$.  Since $\phi^\prime$ has connected fibres (hence it cannot restrict to a finite cover on any subcurve), either the image of $E_x$ is an arithmetic genus one subcurve of level $l$ as well, or it is contracted. But in the first case $p_1$ would have to cleave to $E_x$, which cannot be the case.
 
So, if $\phi^\prime$ does not contract $E_x$, we can assume that $x$ has $l$ distinct rational branches $R_1,\ldots,R_l$, such that $p_1$ cleaves to $R_1$, and $R_l$ is the beginning of a bridge $B_1=R_l,B_2,\ldots,B_h$ towards another (disjoint) genus one subcurve $E_y$. By \eqref{cond:lev1} of $m$-stability for $\mathcal C'_0$, the morphism $\phi^\prime$ has to contract $E_y$ and all of the $B_j$, so that $\phi^\prime(E_y)$ is an elliptic $l^\prime$-fold point of $\mathcal C'_0$ having $\phi^\prime(E_x)$ as a branch. But then $E_y$ and all the curves contained in a disc of radius $\on{dist}(E_x,E_y)$ around it have level bounded above by $l^\prime\leq m+1$, so $\phi$ has to contract them. But by assumption $\phi$ does not contract $R_l$, which is a contradiction.

We have concluded that $E_x\subseteq\on{Exc}(\phi^\prime)$.
 On the other hand, $\on{Exc}(\phi^\prime)$ cannot be any larger. Indeed, let us notice that by condition \eqref{cond:lev1} of $m$-stability applied to $\mathcal C$, the number of special points on $R_1,\ldots,R_l$ is $l'' > m+1$. The same is true for their preimages in $\mathcal C^{ss}$, call them $\widetilde R_1,\ldots,\widetilde R_l$. If $x^\prime=\phi^\prime(\phi^{-1}(x))$ were a genus one singularity of $\mathcal C^\prime_0$, then the component of $\on{Exc}(\phi^\prime)$ containing $E_x$ would be a strictly larger balanced subcurve of $\mathcal C^{ss}$, therefore it would include all the $\widetilde R_1,\ldots,\widetilde R_l$, and then $x^\prime$ would have at least $l''(>m+1)$ branches, which is not allowed by condition \eqref{cond:sing} of $m$-stability. So far, the argument is the same as in Smyth's paper.
 
 Suppose instead that $x^\prime$ were a genus two singularity. In this case, we would know by condition \eqref{cond:p1} of $m$-stability that $\lambda^\prime$ can be obtained by truncating the function described in the existential part of this proof. $\mathcal C^{ss}_0$ either contains another genus one subcurve $E_y$, or it contains a rational bridge between two points of $E_x$; call $Z$ this portion of the curve. If $p_1$ cleaves to $Z$, then $\lambda^\prime$ looks exactly like the distance from $E_x$ near $E_x$, so in particular $E_x\subsetneq {Exc}(\phi^\prime)$ implies that ${Exc}(\phi^\prime)$ contains $\widetilde R_1,\ldots,\widetilde R_l$. If instead $p_1$ cleaves to $E_x$, then $\lambda(Z)\geq\lambda(E_x)$, so the above conclusion is all the more implied.
 
 If $\mathcal C_0$ has two elliptic singularities $x$ and $y$ with a common branch $R$, either $p_1$ cleaves to only one of the two (say $x$), so $y$ has to satisfy the level condition \eqref{cond:lev1} (so we see as above that $E_y$ is a component of $\on{Exc}(\phi')$, and so is $E_x$); or $p_1$ cleaves to the common branch $R$, then we see as above that $E_x\sqcup E_y\subseteq \on{Exc}(\phi')$, but since the genus two core of $\mathcal C_0$ satisfies the level condition \eqref{cond:lev2} it is easy to see that no larger subcurve can be contracted.
 
 Finally, if $\mathcal C_0$ has a genus two singularity, $\on{Exc}(\phi)\subseteq \on{Exc}(\phi')$ by the usual level argument (condition \eqref{cond:lev2} for $\mathcal C'_0$). On the other hand, $p_1$ must cleave to the special component of  $\mathcal C_0$, so $\lambda$ is a cutoff of the function described in the existential part of this proof, and enlarging the exceptional locus (i.e. lowering the cutoff level) would produce a singularity with too many branches (by condition \eqref{cond:lev2} for $\mathcal C_0$ and against condition \eqref{cond:sing} for $\mathcal C'_0$). We conclude that $\on{Exc}(\phi)=\on{Exc}(\phi')$.
\end{proof}

\begin{exa}
 We illustrate the above proof by means of an example, see Figure \ref{fig:example_properness}. Suppose the central fibre has two elliptic subcurves separated by a rational bridge, each of them connected to a two-pointed rational tail. Suppose furthermore that the model has regular total space; thus, every finite edge of the tropicalization has length $1$. The rational tail supporting $p_1$ is attached to a point of $E_1$ that is $2$-torsion with respect to the other elliptic curve. The picture on the left displays the various cutoff levels $\rho$ depending on a choice of $m$. On the right, a cartoon picture of the corresponding $m$-stable limits.
 
  \begin{figure}[htb]
  \centering
  \begin{minipage}[l]{.45\textwidth}
   \centering
   \begin{tikzpicture}
   \coordinate (O) at (0,0);
   \coordinate (E1) at (-1,-1);
   \coordinate (E2) at (1,1);
   \coordinate (R1) at (-2,-4);
   \coordinate (RO) at (.5,-.5);
   \coordinate (R2) at (1.5,.5);
   \draw (E1)node[left]{\tiny$E_1$} -- (O) (E1)--(R1) (O)--(RO) (O)--(E2) (E2)node[left]{\tiny$E_2$}--(R2) (R1)--(-2.1,-5)node[below]{\tiny$p_1$} (R1)--(-1.9,-5) (RO)--(.3,-5) (RO)--(.7,-5) (R2)-- (1.3,-5) (R2)--(1.7,-5); 
   
   \foreach \x in {O,RO,R1,R2}
   \fill[black] (\x) circle (2pt);
   
   \foreach \x in {E1,E2}
   \draw[fill=white] (\x) circle(2pt);
   
   \draw[dashed,gray] (-2.1,.5) -- (1.6,.5) node[right]{\tiny$m=1$};
   \draw[dashed,gray] (-2.1,0) -- (1.6,0) node[right]{\tiny$m=2$};
   \draw[dashed,gray] (-2.1,-.5) -- (1.6,-.5) node[right]{\tiny$m=3$};
   \draw[dashed,gray] (-2.1,-1) -- (1.6,-1) node[right]{\tiny$m=4$};
   \draw[dashed,gray] (-2.1,-4) -- (1.6,-4) node[right]{\tiny$m=5$};
   \end{tikzpicture}
   \end{minipage}
   \begin{minipage}[r]{.45\textwidth}
    \includegraphics[width=.9\textwidth]{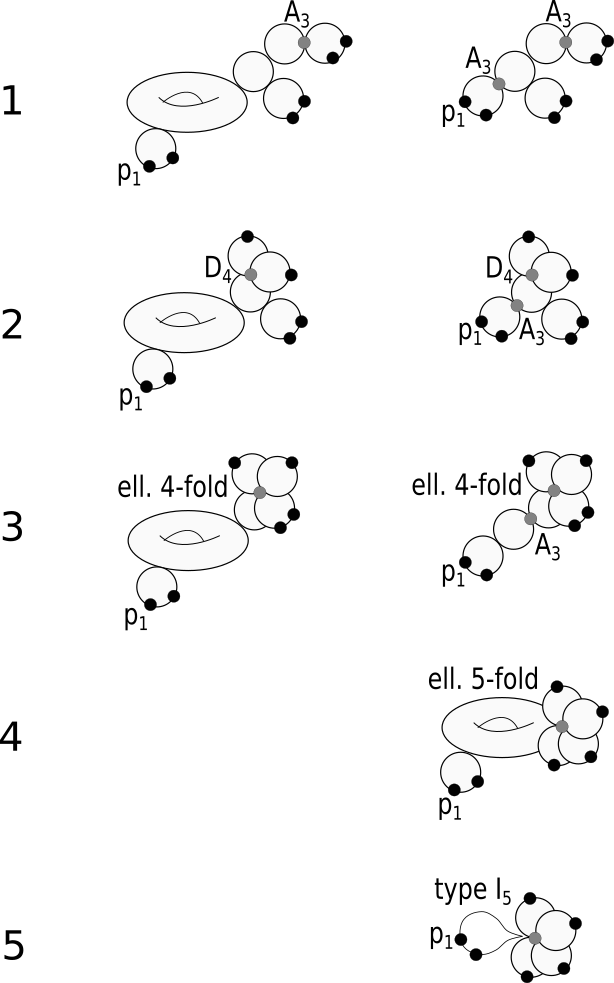}
   \end{minipage}
   \caption{Valuative criterion: an example with two genus one subcurves. Left: the graph of $\lambda$ with values of $\rho$ depending on $m$. Right: $m$, the first contraction $C^\prime$, the end result $\bar C$.}
   \label{fig:example_properness}
  \end{figure}

\end{exa}

\medskip

 \appendix
 
 \section{Crimping spaces}\label{sec:crimp}
 The crimping space parametrises singularities of a given type and pointed normalisation. Knowing it will help us analysing the birational map between two compactifications of $\pazocal{M}_{2,n}$.

 \noindent We recall some concepts from F. van der Wyck's thesis.
Working over $\k$, he considers the stacks:
\begin{itemize}[leftmargin=.5cm]
 \item $\mathcal S$ of reduced one-dimensional (1d) $\k$-algebras $R$,
 \item $\mathcal T$ of reduced 1d algebras with resolution $(R\hookrightarrow (S,J))$, where $S$ is a smooth one-dimensional $\k$-algebra, and $J$ the radical of the conductor of $R\subseteq S$.
\end{itemize}
  Basically, $R$ represents the (local) ring of a reduced curve with one singular point, $S$ is its normalisation, and $J$ is the ideal of the reduced fibre over the singular point of $\operatorname{Spec}(R)$. $\mathcal S$ and $\mathcal T$ are limit-preserving stacks over $\operatorname{Spec}(\k)$ \cite[Proposition 1.21]{vdW}. Furthermore, we may fix a reduced 1d algebra with resolution $\tau_0:(R_0\hookrightarrow(S_0,J_0))$, and consider the substack $\mathcal T(\tau_0)$ of reduced 1d algebras with singularity type $\tau_0$ (i.e. isomorphic to $\tau_0$ locally on both the base and the curve, see \cite[Definition 1.64]{vdW}; that various notions of ``locally'' coincide is proved in \cite[Proposition 1.50]{vdW}). There is a forgetful morphism $\mathcal T\to\mathcal S$, and the \emph{crimping space} of $\tau_0$ is defined to be the fibre over $R_0$ of the restriction of this morphism to $\mathcal T(\tau_0)$. The crimping space is a smooth $\k$-scheme \cite[Theorems 1.70 and 1.73]{vdW}; indeed, it is isomorphic to the quotient of $\Aut_{(S_0,J_0)/\k}$ by $\Aut_{(S_0,J_0)/R_0}$, the latter consisting of automorphisms of the normalisation that preserve the subalgebra of the singularity; moreover, by \cite[Theorem 1.53]{vdW} the quotient can be computed after modding out the lowest power of $J$ contained in $R$, denoted by $\Aut_{(S,J)}^{\mod J^k}$ respectively $\Aut_{(S,J)/R}^{\mod J^k}$. Crimping spaces can be thought of as moduli for the normalisation map.
  
\begin{lem}\label{lem:crimping}
 If $\operatorname{char}(\k)\neq2,3$, the crimping space of a genus two singularity of type $I$ (resp. $I\!I$) with $m$ branches is the disjoint union of $m$ (resp. ${m}\choose{2}$) copies of $\Aaff^1\times(\Aaff^1\setminus\{0\})^{m-1}$.
\end{lem}
\begin{proof}
We resume notation from the previous section. We are going to fix the subalgebra $\tau_0$ given in coordinates by \eqref{coordIII} and \eqref{coordII} respectively.

\textbf{Type $I$}: recall that in this case $\tm^4\subseteq R$. For a $\k$-algebra $A$, let
\[G_i(A)=\{t_i\mapsto g_{i1}t_i+g_{i2}t_i^2+g_{i3}t_i^3,t_j\mapsto t_j\ |\ g_{i1}\in A^\times,g_{i2},g_{i3}\in A\}.\]

Suppressing $i$ from the notation, with respect to the standard basis $\langle 1,t,t^2,t^3\rangle$ of $\k[t]/(t^4)$, the action of $(g)$ is represented by the following matrix:
\begin{equation*}
\begin{pmatrix}
1 & {} & {} & {} \\
{} & g_1 & {} & {} \\
{} & g_2  & g_1^2  & {}  \\
{} & g_3 & 2g_1g_2 & g_1^3
\end{pmatrix}
\end{equation*}
from which we see that $G$ is a semidirect product (split extension) of the multiplicative group $\Gm$ with a group $H$, which is a subgroup of the Heisenberg group and itself a non-split extension of two copies of the additive group:
\[1\to \Ga\to H\to \Ga\to 1\]

Now, for the pointed normalisation, the automorphism group is
\[\Aut_{(\tR,\tm)}^{\mod\tm^4}(A)= \mathfrak{S}_m \ltimes (G_1\times\ldots\times G_m)(A).\]

Consider now the action of a group element of the form $(\id_{\mathfrak{S}_m};g_1,\ldots,g_m)$ on the given generators of $R$:
\begin{align*}
 x_i\mapsto& \ldots\oplus g_{i1}t_i+g_{i2}t_i^2+g_{i3}t_i^3\oplus\ldots\oplus g_{m1}^3t_m^3,\quad\text{for } i=1,\ldots,m-1;\\
 x_m\mapsto& \ldots\oplus g_{m1}^2t_m^2+2g_{m1}g_{m2}t_m^3 \pmod{\tm^4}.
\end{align*}
The former belongs to $R$ iff $g_{i1}=g_{m1}^3$; the latter does iff $g_{m2}=0$. These elements span a subgroup isomorphic to $(H^{m-1}\times(\Gm\ltimes\Ga))(A)$. On the other hand, there is a special (singular) branch, parametrised by $t_m$. We conclude that
\[\Aut_{\tau_0}^{\mod\tm^4}(A)=\mathfrak{S}_{m-1}\ltimes(H^{m-1}\times(\Gm\ltimes\Ga))(A).\]
The quotient is therefore isomorphic to $m$ copies of $\Aaff^1\times(\Aaff^1\setminus\{0\})^{m-1}$. 

\textbf{Type $I\!I$}: recall that in this case $\tm^3\subseteq R$. For a $\k$-algebra $A$, let
\[G_i(A)=\{t_i\mapsto g_{i1}t_i+g_{i2}t_i^2,t_j\mapsto t_j\ |\ g_{i1}\in A^\times,g_{i2}\in A\},\]
 so $G_i= \Gm\ltimes\Ga$, and notice that
\[\Aut_{(\tR,\tm)}^{\mod\tm^3}(A)= \mathfrak{S}_m\ltimes(G_1\times\ldots\times G_m)(A).\]

Consider now the action of a group element of the form $(\id_{\mathfrak{S}_m};g_1,\ldots,g_m)$ on the given generators of $R$:
\begin{align*}
 x_i\mapsto& \ldots\oplus g_{i1}t_i+g_{i2}t_i^2\oplus\ldots\oplus g_{m1}^2t_m^2,\quad\text{for } i=2,\ldots,m-1;\\
 x_1\mapsto& g_{11}t_1+g_{12}t_1^2\oplus\ldots\oplus g_{m1}t_m+g_{m2}t_m^2 \pmod{\tm^3}.
\end{align*}
The former belongs to $R$ iff $g_{i1}=g_{m1}^2$; the latter does iff $g_{11}=g_{m1}$ and $g_{12}=g_{m2}$. These elements span a subgroup isomorphic to $\Gm\ltimes\Ga^{m-1}(A)$. On the other hand, all branches are smooth (therefore, isomorphic to each other), but two of them (parametrised by $t_1$ and $t_m$ respectively) are tangent, thus forming a distinguished pair. We conclude that
\[\Aut_{\tau_0}^{\mod\tm^3}(A)=(\mathfrak{S}_2\times\mathfrak{S}_{m-2})\ltimes(\Gm\ltimes\Ga^{m-1})(A).\]
The quotient is then isomorphic to $\binom{m}{2}$ copies of $\Aaff^1\times(\Aaff^1\setminus\{0\})^{m-1}$.
\end{proof}

\begin{rem}\label{rem:characteristic}
 The restrictions on the characteristic of the base field in Lemmas \ref{lem:aut} and \ref{lem:crimping} rule out the sporadic occurrence of infinite families of automorphisms, and its effect on the crimping spaces. For example, when $\operatorname{char}(\k)=2$, the singularities $\k[\![t^2,t^5]\!]$ and $\k[\![t^2+t^3,t^4,t^5]\!]$ are not isomorphic, the group of infinitesimal automorphisms has positive dimension, and the crimping space consists of an isolated point \cite[Examples 1.79-80]{vdW}.
\end{rem}

Once the special branch(es) has been fixed, we can identify the crimping space of the type $I$ (resp. $I\!I$) singularity with the parameters $(\gamma_{i,m})_{i=1,\ldots,m}\in(\k^\times)^{m-1}\times\k$ (resp. $(\alpha_{1,m},\beta_{i,m})_{i=1,\ldots,m}\in \k^\times\times\k\times(\k^\times)^{m-1}$) appearing in the expression \eqref{coordIII-cs} (resp. \eqref{coordII-cs}) for the generators of the singularity subalgebra.

%The benefit of a two-step classification should now be clear: if we do not allow ourselves to change coordinates (i.e. act by automorphisms of the normalisation) until the end, the crimping space appears in the expressions \eqref{coordII-cs} and \eqref{coordIII-cs} for the generators of the singularity subalgebra.

There is a more geometric way to realise the crimping spaces. It is well-known that an ordinary cusp of genus one can be obtained by collapsing (\emph{push-out}) any non-zero tangent vector at $p\in\Aaff^1$. More generally, a Gorenstein singularity of genus one and $m$ branches can be obtained by collapsing a generic (not contained in any coordinate linear subspace) tangent line at an ordinary $m$-fold point (a non-Gorenstein singularity of genus zero) \cite[Lemma 2.2]{SMY1}. Therefore, the crimping space of the elliptic $m$-fold point, which is isomorphic to $(\Aaff^1\setminus\{0\})^{m-1}$, can be realised as the complement of the coordinate hyperplanes inside $\PP(T_pR_m)\simeq\PP^{m-1}$, where $(R_m,p)$ is the rational $m$-fold point. Besides, this gives rise to a natural compactification of the crimping space supporting a universal family of curves - in fact, two: either we collapse non-generic tangent vectors, obtaining non-Gorenstein singularities along the boundary (this family $\mathcal C$ admits a common (semi)normalisation by the trivial family $\widetilde{\mathcal C}=R_m\times \PP(T_pR_m)$); or we blow $\widetilde{\mathcal C}$ up along the boundary (\emph{sprouting}), so that the non-Gorenstein singularities are replaced by elliptic $m$-fold points having strictly semistable branches \cite[\S 2.2-3]{SMY2}.

Similarly, a Gorenstein singularity of genus two can be obtained by collapsing a generic line in the tangent space of a non-Gorenstein singularity of genus one. Indeed, $\tau_0^{I}$ admits a partial normalisation by $\sigma_0^{I}$, which is the decomposable union of a cusp (parametrised by $t_m$) together with $m-1$ axes; the local ring of $\sigma_0^{I}$ is obtained from that of $\tau_0^{I}$ by adjoining the generator $t_m^3$.  $\tau_0^{I\!I}$ admits a partial normalisation by $\sigma_0^{I\!I}$, which is the decomposable union of a tacnode in the $(t_1,t_m)$-plane together with $m-2$ axes, adjoining the generator $t_m^2$.

These fit together nicely in a unifying picture: if we restrict $\mathcal C$ from the previous paragraph to the union of the coordinate lines in $\PP(T_pR_m)$, we obtain $m$ copies of $\sigma_0^{I}$ over the coordinate points, together with $\binom{m}{2}$ copies of the universal curve of type $\sigma_0^{I\!I}$ over its crimping space - which is isomorphic to $\Aaff^1\setminus\{0\}$ - identified with the line minus two points. Let $P=\PP(T_{\mathcal C/\PP,p|\cup\text{lines}})$ be the projectivised tangent space of the fibre at the singular point. For each of the $\binom{m}{2}$ coordinate lines, $P$ has one component $P^{I\!I}_i$ that is a $\PP^{m-1}$-bundle over the line; besides, $P$ has $m$ components $P^{I}_j$ isomorphic to $\PP^m$ and supported over the points. The crimping space of the genus two singularities with $m$ branches (of type $I$ and $I\!I$ together) can be realised as an open subscheme of $P$, obtained by removing from the $\PP^{m-1}$-fibres of $P^{I\!I}$ the $m-1$ hyperplanes generated by (a) the tangent cone of the tacnode and the $m-2$ axes, and (b) the plane containing the tacnode and all but one of the $m-2$ axes; and from each $P^{I}_j$ the $m$ planes generated by (a) the tangent cone of the cusp and the $m-1$ axes, and (b) the plane containing the cusp and all but one of the $m-1$ axes.

\smallskip

Finally, we want to describe another point of view on the dichotomy between the atom and the non-atom (see Definition \ref{def:atom}). We once again recall some relevant concepts from van der Wyck's thesis. The notion of \emph{type} of a (proper, reduced) pointed curve \cite[Definition 1.87]{vdW} is a generalisation of the dual graph of a nodal curve, where any kind of reduced curve singularity is allowed, an incidence relation records the branches meeting in a given singular point, another map tells us which branches belong to the same irreducible component, and the genus of the latter. Let $\mathcal N_T$ parametrise curves of type $T$ together with a resolution (a finite birational morphism from a smooth pointed curve, where the preimage of the singularities is marked as well; see \cite[Definitions 1.95 and 1.100]{vdW} for more details). Then $\mathcal N_T$ admits a map to the stack of all curves (forgetting the resolution), and a map to the stack of (not necessarily connected) smooth pointed curves of the associated type $\mathcal M_T$ (forgetting the singular curve); the latter is a product of stacks of the form $\pazocal M_{g,n}$ modulo the (finite) automorphism group of the type $T$. Van der Wyck proves that $\mathcal N_T\to\mathcal M_T$ is a locally trivial fibration in the \'etale topology, the fibre of which is nothing but the product of the crimping spaces of all the singularities appearing in $T$; therefore $\mathcal N_T$ is an algebraic stack as well \cite[Theorem 1.105 and Corollary 1.106]{vdW}.

In case $T$ consists of a unique Gorenstein singularity of genus two, with $m$ one-marked rational branches, it is not hard to see that the stack $\mathcal N_T$ is isomorphic to $[\Aaff^1/\Gm]$ (see \cite[Examples 1.111-112]{vdW} for the $I_1$ and $I\!I_2$ cases), so it has two points: one with $\Gm$, and the other one with trivial stabiliser, corresponding to the atom and non-atom respectively.

Again, there is a more geometric way to realise the dichotomy. The non-Gorenstein genus one singularity of type $\sigma_0^{I\!I}$ (resp. $\sigma_0^{I}$), with one-marked rational branches, has automorphism group $\Gm^{m-1}$ (resp. $\Gm^m$). This acts on the tangent space at the singular point: of the lines fixed by this action, only one (call it $\ell^\prime$) sits inside the open subset corresponding to the crimping space; all other lines in the crimping space are identified under the group action (call $\ell$ their equivalence class), i.e. the action of the automorphism group on the crimping space has two orbits, $\ell^\prime$ with stabiliser $\Gm$, and $\ell$ with trivial stabiliser. Collapsing $\ell$ yields the non-atom, while collapsing $\ell^\prime$ yields the atom.

%As a third viewpoint, automorphisms can be studied by twisting the exact sequences of Lemma \ref{lem:aut} by the ideal of the markings, and then taking global sections. The dicotomy arises then from the map $\phi$: if the last condition imposed on infinitesimal automorphisms interweaves first and second order non-trivially (i.e. when $\beta_{1,m}$, resp. $\gamma_{m,m}$, are non-zero) then it is enough that automorphisms are trivial to second order on every branch for them to be trivial for good.

\bibliographystyle{alpha}
\bibliography{genus_two} 
 
\begin{comment}

\end{comment}

\end{document}